\documentclass[12pt]{amsart}
\usepackage[margin=1in]{geometry}
\usepackage{amssymb,amsfonts,amsmath}
\usepackage{color}
\usepackage{soul}
\usepackage{enumerate}
\usepackage{mathrsfs}
\usepackage{hyperref}
\usepackage[capitalise]{cleveref}
\usepackage{constants}   

\newtheorem{theorem}{Theorem}[section]
\newtheorem{corollary}[theorem]{Corollary}

\newtheorem{proposition}[theorem]{Proposition}
\newtheorem{lemma}[theorem]{Lemma}

\newtheorem{question*}{Question}
\newtheorem{problem*}{Problem}

\theoremstyle{definition}

\theoremstyle{remark}
\newtheorem*{remark}{Remark}
\newtheorem*{remarks}{Remarks}

\numberwithin{equation}{section}

\crefname{figure}{Figure}{Figures}
\theoremstyle{plain}
\newtheorem*{theorem*}{Theorem}
\crefname{theorems}{Theorem}{Theorems}
\crefname{corollaries}{Corollary}{Corollaries}
\newtheorem*{corollary*}{Corollary}
\crefname{corollaries*}{Corollary}{Corollaries}
\crefname{lemma}{Lemma}{Lemmas}
\crefname{proposition}{Proposition}{Propositions}
\crefname{conjectures}{Conjecture}{Conjectures}
\newtheorem*{conjonjecture*}{Conjecture}
\crefname{conjonjectures*}{Conjecture}{Conjectures}
\crefname{definitions}{Definition}{Definitions}
\crefname{hypotheses}{Hypothesis}{Hypotheses}

\newcommand{\Z}{\mathbb{Z}}
\newcommand{\R}{\mathbb{R}}
\newcommand{\Q}{\mathbb{Q}}

\newcommand{\kn}{\mathfrak{n}}
\newcommand{\kd}{\mathfrak{d}}
\newcommand{\kp}{\mathfrak{p}}

\newcommand{\re}{\textup{Re}}

\newcommand{\kb}{\mathfrak{b}}

\renewcommand{\tilde}{\widetilde}

\newcommand{\GL}{\mathrm{GL}}

\newcommand{\N}{\mathrm{N}}

\makeatletter
\DeclareFontFamily{U}  {MnSymbolF}{}
\DeclareSymbolFont{symbolsMN}{U}{MnSymbolF}{m}{n}
\SetSymbolFont{symbolsMN}{bold}{U}{MnSymbolF}{b}{n}
\DeclareFontShape{U}{MnSymbolF}{m}{n}{
    <-6>  MnSymbolF5
   <6-7>  MnSymbolF6
   <7-8>  MnSymbolF7
   <8-9>  MnSymbolF8
   <9-10> MnSymbolF9
  <10-12> MnSymbolF10
  <12->   MnSymbolF12}{}
\DeclareFontShape{U}{MnSymbolF}{b}{n}{
    <-6>  MnSymbolF-Bold5
   <6-7>  MnSymbolF-Bold6
   <7-8>  MnSymbolF-Bold7
   <8-9>  MnSymbolF-Bold8
   <9-10> MnSymbolF-Bold9
  <10-12> MnSymbolF-Bold10
  <12->   MnSymbolF-Bold12}{}
\DeclareMathSymbol{\tbigtimes}{\mathop}{symbolsMN}{2}
\newcommand*{\bigtimes}{%
  \DOTSB
  \tbigtimes
  \slimits@ 
}
\makeatother

\newcommand{\ka}{\mathfrak{a}}

\renewcommand{\bar}{\overline}

\renewcommand{\epsilon}{\varepsilon}

\renewcommand{\pmod}[1]{\, (\mathrm{mod} {\, #1})}

\newcommand{\kq}{\mathfrak{q}}

\renewcommand{\pmod}[1]{\left(\mathrm{mod}\,\,#1\right)}

\newconstantfamily{abcon}{symbol=c}

\makeatletter
\let\@wraptoccontribs\wraptoccontribs
\makeatother

\title[]{An unconditional $\GL_n$ large sieve}

\author{Jesse Thorner}
\address{Department of Mathematics, University of Illinois, Urbana, IL 61801}
\email{jesse.thorner@gmail.com}
\author{Asif Zaman}
\address{Department of Mathematics, University of Toronto, Toronto, Ontario, Canada M5S 2E4}
\email{zaman@math.toronto.edu}

\begin{document}

\begin{abstract}
Let $\mathfrak{F}_n$ be the set of cuspidal automorphic representations $\pi$ of $\mathrm{GL}_n$ over a number field with unitary central character.  We prove two unconditional large sieve inequalities for the Hecke eigenvalues of $\pi\in\mathfrak{F}_n$, one on the integers and one on the primes. The second leads to the first unconditional zero density estimate for the family of $L$-functions $L(s,\pi)$ associated to $\pi\in\mathfrak{F}_n$, which we make log-free.  As an application of the zero density estimate, we prove a hybrid subconvexity bound for $L(\frac{1}{2},\pi)$ for a density one subset of $\pi\in\mathfrak{F}_n$.
\end{abstract}

\maketitle

\section{Introduction and statement of the main results}

Let $\mathbb{A}_F$ be the ring of adeles over a number field $F$ with ring of integers $\mathcal{O}_F$, absolute norm $\N=\N_{F/\Q}$, and absolute discriminant $D_F$.  For an integer $n\geq 1$, let $\mathfrak{F}_n$ be the universal family of all cuspidal automorphic representations $\pi$ of $\mathrm{GL}_n(\mathbb{A}_F)$ with unitary central character, which is normalized to be trivial on the diagonally embedded copy of the positive reals so that $\mathfrak{F}_{n}$ is discrete (see \cite[Corollary 9]{Brumley}).  To each $\pi\in\mathfrak{F}_n$, Iwaniec and Sarnak \cite{IS} associate an analytic conductor $C(\pi)\geq 1$ measuring the arithmetic and spectral complexity of $\pi$ (see \eqref{eqn:analytic_conductor_def}).  We consider the truncated family $\mathfrak{F}_{n}(Q)=\{\pi\in\mathfrak{F}_n\colon C(\pi)\leq Q\}$.  For $\pi\in\mathfrak{F}_n$, let $\kq_{\pi}$ be the arithmetic conductor, and let $L(s,\pi)=\sum_{\kn}\lambda_{\pi}(\kn)\N\kn^{-s}$ be the standard $L$-function of $\pi$.  Here, $\lambda_{\pi}(\kn)$ is the Hecke eigenvalue of $\pi$ at a nonzero integral ideal $\kn$ of $\mathcal{O}_F$.

If $n=1$ and $F=\Q$, then $\pi\in\mathfrak{F}_1(Q)$ corresponds with a primitive Dirichlet character $\chi$ to modulus $q_{\chi}\leq Q$ with $\lambda_{\pi}(m)=\chi(m)$, and $|\mathfrak{F}_1(Q)|\asymp Q^2$.  For the discussion that follows, we abuse notation and write $\chi\in\mathfrak{F}_1(Q)$.  The classical large sieve inequality for Dirichlet characters states that if $M,N,Q\geq 1$ and $a:\Z\to\mathbb{C}$ is a function, then
\begin{equation}
	\label{eqn:large_sieve_GL1}
	\sum_{\chi\in\mathfrak{F}_1(Q)}\Big|\sum_{m=M+1}^{M+N}\chi(m)a(m)\Big|^2\ll(N+|\mathfrak{F}_1(Q)|)\sum_{m=M+1}^{M+N}|a(m)|^2.
\end{equation}
(See \cite[(7.31)]{IK}.) This improves on the trivial bound $N|\mathfrak{F}_1(Q)|\sum |a(m)|^2$ that follows from the Cauchy--Schwarz inequality.  The large sieve serves as a quasi-orthogonality statement for characters to varying moduli, leading to powerful substitutes for the generalized Riemann hypothesis (GRH) for Dirichlet $L$-functions like the Bombieri--Vinogradov theorem.  When $n\geq 2$ or $F\neq\Q$, minor changes to work of Duke and Kowalski \cite[Section 4]{DK} show that if $\alpha\geq\frac{1}{2}$, $\epsilon>0$, and $a$ is a complex-valued function on the integral ideals of $\mathcal{O}_F$, then under the generalized Ramanujan conjecture (GRC) (which implies that $|\lambda_{\pi}(\kn)|\ll_{n,\epsilon}\N\kn^{\epsilon}$), we have
\begin{equation}
\label{eqn:auto_large_sieve}
\sum_{\pi\in\mathfrak{F}_{n}(Q)}\Big|\sum_{\N\kn\leq N}\lambda_{\pi}(\kn)a(\kn)\Big|^2\ll_{n,[F:\Q],\epsilon} (NQ)^{\epsilon}(N + N^{\alpha}Q^{n(1-\alpha)}|\mathfrak{F}_{n}(Q)|)\sum_{\N\kn\leq N}|a(\kn)|^2.
\end{equation}
For $n\leq 4$, Brumley \cite[Corollary 3]{Brumley_2} unconditionally proved \eqref{eqn:auto_large_sieve} with $\alpha=1-(n^2+1)^{-1}$.  For $n=4$, his proof uses the automorphy of the exterior square lift from $\GL_4$ to $\GL_6$ \cite{Kim}.  The ideas in \cite[Section 7.1]{IK} suggest that one can replace $N^{\alpha}Q^{n(1-\alpha)}$ with 1, matching \eqref{eqn:large_sieve_GL1}.

We prove a variant of \eqref{eqn:auto_large_sieve} which holds for all $n\geq 1$ without recourse to unproven progress towards GRC.  This appears to be the first unconditional large sieve for $\mathfrak{F}_n(Q)$ when $n\geq 5$.

\begin{theorem}
\label{thm:large_sieve}
	Fix $\epsilon>0$, $n\geq 1$.  If $N,Q\geq 1$ and $a(\kn)$ is a complex-valued function, then
	\[
	\sum_{\pi\in\mathfrak{F}_{n}(Q)}\Big|\sum_{\substack{\N\kn\leq N \\ (\kn,\kq_{\pi})=\mathcal{O}_F}}\lambda_{\pi}(\kn)a(\kn)\Big|^2\ll_{n,[F:\Q],\epsilon}(NQ)^{\epsilon}(N+Q^{n^2+n}|\mathfrak{F}_{n}(Q)|)\sum_{\N\kn\leq N}|a(\kn)|^2.
	\]
\end{theorem}
\begin{remarks}~
\begin{enumerate}[1.]
\item  There exists a constant $c_{n,[F:\Q]}>0$ such that for all $\epsilon>0$, we have the bounds
\begin{equation}
\label{eqn:poly_upper}
	Q^{n+1}D_F^{-n(n+\frac{1}{2})}(c_{n,[F:\Q]}+O_{n,F}((\log Q)^{-1}))\leq|\mathfrak{F}_{n}(Q)|\ll_{n,[F:\Q],\epsilon}(Q^{2n}D_F^{-n^2})^{1+\epsilon}.
\end{equation}
Brumley, Thorner, and Zaman \cite[Theorem A.1]{BTZ} proved the upper bound.  Brumley and Mili{\'c}evi{\'c} \cite[Theorem 1.1]{BM} proved the lower bound, which is the expected asymptotic.
\item  Our proof is easily modified to accommodate subfamilies of $\mathfrak{F}_n(Q)$.
\item  Under GRC and the generalized Lindel{\"o}f hypothesis, the $Q^{n^2+n}$ improves to $\min\{Q^n,\sqrt{N}\}$.  This seems to be the limit of our method.  Under GRC alone, the $Q^{n^2+n}$ improves to $Q^n$.
\item  See \cref{sec:stratagem} for an overview of our proof of \cref{thm:large_sieve}.  The ideas are encapsulated in the proof of a new inequality for Dirichlet coefficients of Rankin--Selberg $L$-functions (\cref{prop:RSbound}), which might be of independent interest.
\end{enumerate}
\end{remarks}

We see from \eqref{eqn:poly_upper} that \cref{thm:large_sieve} is sharp for $N>Q^{n^2+3n}$.  Despite this weak range, \cref{thm:large_sieve} is useful for studying zeros of $L$-functions.  For each $\pi\in\mathfrak{F}_n$, we expect $L(s,\pi)$ to satisfy the generalized Riemann hypothesis (GRH):  if $\re(s)>\frac{1}{2}$, then $L(s,\pi)\neq 0$.  Since GRH remains open, and existing zero-free regions can be limiting in applications, it is useful to show that few zeros of $L(s,\pi)$ lie near the line $\re(s)=1$.  Hence we define
\begin{equation}
\label{eqn:zero_density_count}
N_{\pi}(\sigma,T)=\#\{\rho=\beta+i\gamma\colon L(\rho,\pi)=0,~\beta>\sigma,~|\gamma|\leq T\}.
\end{equation}
Corresponding with $F=\Q$ and $n=1$, Montgomery \cite{Montgomery} proved that if $Q,T\geq 1$, then
\begin{equation}
\label{eqn:Montgomery_ZDE}
\sum_{\chi\in\mathfrak{F}_1(Q)}N_{\chi}(\sigma,T)\ll (|\mathfrak{F}_1(Q)|T)^{\min\{\frac{3}{2-\sigma},\frac{2}{\sigma}\}(1-\sigma)}(\log QT)^{13},\qquad \sigma\geq \frac{1}{2}.
\end{equation}
Thus a vanishingly small proportion of zeros of Dirichlet $L$-functions lie near $\re(s)=1$.  As part of his proof that the least prime $p\equiv a\pmod{q}$ is at most $q^{O(1)}$ (GRH replaces $O(1)$ with $2+o(1)$),  Linnik \cite{Linnik} developed powerful results for the distribution of zeros of Dirichlet $L$-functions, including a {\it log-free} zero density estimate.  Gallagher \cite{Gallagher} and Jutila \cite{Jutila} unified the work of Montgomery and Linnik.  Jutila proved that if $Q,T\geq 1$, then for all $\epsilon>0$,
\begin{equation}
\label{eqn:Gallagher_LFZDE}
\sum_{\chi\in\mathfrak{F}_1(Q)}N_{\chi}(\sigma,T)\ll_{\epsilon} (|\mathfrak{F}_1(Q)|T)^{(2+\epsilon)(1-\sigma)},\qquad \sigma\geq \frac{4}{5}.
\end{equation}

Kowalski and Michel \cite[Theorem 2]{KM} extended Jutila's work using the ideas of Duke and Kowalski. For $F=\Q$, it follows from the work in \cite{KM} that if each $\pi\in\mathfrak{F}_{n}(Q)$ satisfies\footnote{In fact, they only require (at the cost of a larger exponent) that $\theta_n\leq \frac{1}{4}-\delta$ for some fixed $\delta>0$, where $\theta_n$ is given by \eqref{eqn:LRS_finite}.  See \cref{sec:stratagem}.} GRC, then for some constant $A_n>0$, we have that (see also \cite{an2020logfree})
\begin{equation}
\label{eqn:KM_LFZDE}
\sum_{\pi\in\mathfrak{F}_{n}(Q)}N_{\pi}(\sigma,T)\ll_{n} T^{A_n} (Q^{\frac{5n}{2}}|\mathfrak{F}_n(Q)|)^{\frac{1-\sigma}{2\sigma-1}},\qquad \sigma\geq\frac{3}{4}.
\end{equation}
Much like \eqref{eqn:Montgomery_ZDE} and \eqref{eqn:Gallagher_LFZDE}, an estimate such as \eqref{eqn:KM_LFZDE} often suffices to prove results (most often on average over families) which are commensurate with what GRH predicts.  Using Brumley's work \cite{Brumley_2}, one can prove \eqref{eqn:KM_LFZDE} when $n\leq 4$ without recourse to  unproven progress towards GRC (with a worse exponent).  We refine \cref{thm:large_sieve} for $\kn$ restricted to {\it prime} ideals of large norm (see \cref{thm:pre_large_sieve}) to prove the first unconditional zero density estimate for the sum in \eqref{eqn:KM_LFZDE}, which we make log-free using the ideas of Gallagher \cite{Gallagher}  and Soundararajan and Thorner \cite{ST} rather than \cite{Jutila,KM}.
\begin{theorem}
	\label{thm:LFZDE}
	Fix $n\geq 1$.  If $N_{\pi}(\sigma,T)$ is given by \eqref{eqn:zero_density_count}, then for $Q,T\geq 1$ and $0\leq\sigma\leq 1$,
	\[
	\sum_{\pi\in\mathfrak{F}_{n}(Q)}N_{\pi}(\sigma,T)\ll_{n,[F:\Q]}(QT^{[F:\Q]})^{10^7 n^4(1-\sigma)}.
	\]
\end{theorem}
\begin{remark}
	\cref{thm:large_sieve}, along with the ideas in \cite{Montgomery}, leads to a density estimate that is {\it not} log-free but has a much smaller exponent.  The log-free estimate is much more delicate.
\end{remark}

We now describe an application of \cref{thm:LFZDE}.  For $\pi\in\mathfrak{F}_{n}$, we seek bounds for $|L(\tfrac{1}{2},\pi)|$ in terms of $C(\pi)$.  The generalized Lindel{\"o}f hypothesis (a corollary of GRH) asserts that $|L(\tfrac{1}{2},\pi)|\ll_{n,[F:\Q],\epsilon} C(\pi)^{\epsilon}$ for any $\epsilon>0$.  When $F=\Q$ (only for convenience), Soundararajan and Thorner \cite[Corollary 2.7]{ST} proved an unconditional log-free zero density estimate for an individual $L(s,\pi)$, from which the bound
\begin{equation}
\label{eqn:ST_weak}
|L(\tfrac{1}{2},\pi)|\ll_{n} \frac{C(\pi)^{\frac{1}{4}}}{(\log C(\pi))^{1/(10^{17}n^{3})}}
\end{equation}
follows.  Subconvexity bounds of the shape $|L(\tfrac{1}{2},\pi)|\ll_{n,[F:\Q]} C(\pi)^{\frac{1}{4}-\delta}$ for some constant $\delta=\delta_{n,[F:\Q]}>0$ are important in many equidistribution problems.  See \cite{IS,Michel} for further discussion and \cite{Blomer_twist,DFI,HarcosMIchel,Li_Xiaoqing,MV,Munshi,Venk1} for a sample of some amazing progress.

When $n=1$ and $F=\Q$, an application of \eqref{eqn:large_sieve_GL1} using the approximate functional equation \cite[Section 5.2]{IK} shows that a density one subset of Dirichlet $L$-functions satisfy the generalized Lindel{\"o}f hypothesis (see \cite[Theorem 7.34]{IK}, for example).  We cannot adapt this approach to obtain a power-savings over \eqref{eqn:ST_weak} for a density one subset of $\pi\in\mathfrak{F}_n(Q)$ using \cref{thm:large_sieve} because of the poor $Q$-dependence and the condition $(\kn,\kq)=\mathcal{O}_F$, which we do not know how to remove (see \cref{sec:large_sieve}).  Instead, we use \cref{thm:LFZDE} and the bound
\begin{equation}
\label{eqn:STl1/2}
\log|L(\tfrac{1}{2},\pi)|\leq\Big(\frac{1}{4}-\frac{\alpha}{10^{9}}\Big)\log C(\pi) + \frac{\alpha}{10^7} N_{\pi}(1-\alpha,6)+2\log|L(\tfrac{3}{2},\pi)|+O_{n,[F:\Q]}(1)
\end{equation}
for all $0\leq\alpha<\frac{1}{2}$, which follows from \cite[Theorem 1.1]{ST} with minor changes for $F\neq\Q$.


\begin{theorem}
	\label{thm:subconvexity}
	If $\epsilon\geq 0$ and $Q\ggg_{n,F} 1$, then $|L(\tfrac{1}{2},\pi)|\ll_{n,[F:\Q]}C(\pi)^{\frac{1}{4}-\epsilon/(10^{16}n^4)}$ for all except $O_{n,[F:\Q]}(|\mathfrak{F}_n(Q)|^{\epsilon})$ of the $\pi\in\mathfrak{F}_{n}(Q)$.
\end{theorem}

On the same day that this paper was first posted, Blomer \cite[Theorem 4, Corollary 5]{Blomer} posted a preprint showing that if $\mathcal{F}_I(q)$ is the set of $\pi\in\mathfrak{F}_n$ corresponding to Hecke--Maa\ss~forms over $\Q$ of arithmetic conductor $q$ and Laplace eigenvalue in $I\subseteq[0,\infty)$, then
\begin{equation}
\label{eqn:Blomer}
\sum_{\pi\in\mathcal{F}_I(q)}\Big|\sum_{\substack{m\leq N \\ (m,q)=1}}\lambda_{\pi}(m)a(m)\Big|^2\ll_{I,n,\epsilon}|\mathcal{F}_I(q)|q^{\epsilon}\sum_{\substack{m\leq N \\ (m,q)=1}}|a(m)|^2
\end{equation}
once $N\ll_{I,n} q$ with a sufficiently small implied constant.  This improves on work of Venkatesh \cite{Venkatesh}.  Note that \eqref{eqn:Blomer} holds for $\pi$ of a given large arithmetic conductor and small eigenvalue, while \cref{thm:large_sieve} holds as the arithmetic conductor and eigenvalue vary with essentially no constraint.  One cannot sum \eqref{eqn:Blomer} over $q\leq Q$ to recover a result like \cref{thm:large_sieve}.  Also, note that \cref{thm:large_sieve} and \eqref{eqn:Blomer} are sharp in complementary ranges.  The approximate functional equation for $L(\frac{1}{2},\pi)$ and \eqref{eqn:Blomer} imply that for $0<\delta\leq \frac{1}{4}$, all except $O_{n,I,\epsilon}(|\mathcal{F}_I(q)|^{1-\frac{2\delta}{n-1}+\epsilon})$ of the $\pi\in\mathcal{F}_I(q)$ satisfy $|L(\frac{1}{2},\pi)|\ll_{I,n}q_{\pi}^{\delta}$.  \cref{thm:subconvexity} furnishes a much smaller hybrid-aspect power-savings over \eqref{eqn:ST_weak} for all $\pi$ outside of a much smaller exceptional set.\footnote{A few months after this paper was first posted, Jana \cite[Theorems 5 and 6]{2020arXiv200109640J} proved a variant of Blomer's results in \cite{Blomer} in the analytic conductor aspect for automorphic forms on $\mathrm{PGL}_n(\Z)$, all of which have arithmetic conductor $q_{\pi}=1$.  One could similarly contrast our work with his.}

\subsection*{Notation}
The expressions $f\ll_{\nu}g$ and $f=O_{\nu}(g)$ mean that there exists an effectively computable constant $c_{\nu}>0$, depending at most on $\nu$, such that $|f|\leq c_{\nu}|g|$.  The expression $f\ggg_{\nu}g$ means that there exists a {\it sufficiently large} effectively computable constant $c_{\nu}>0$, depending at most on $\nu$, such that $|f|\geq c_{\nu}|g|$.  Sufficiency will depend on context.  We write $(\ka,\kb)$ and $[\ka,\kb]$ for the GCD and LCM of two integral ideals $\ka$ and $\kb$.

\subsection*{Overview of the paper}

In \cref{sec:L-functions}, we recall basic properties of standard and Rankin--Selberg $L$-functions.  In \cref{sec:prelim_large_sieve}, we provide an explicit description of the coefficients of such $L$-functions in terms of $A_{\pi}(\kp)$ and $A_{\pi'}(\kp)$.  We use these explicit descriptions to prove \cref{prop:RSbound}.  In \cref{sec:large_sieve}, after outlining our strategy in \cref{sec:stratagem},  we use these explicit descriptions to prove the large sieve inequalities in \cref{thm:large_sieve,thm:pre_large_sieve}.  In \cref{sec:ZDE}, we use \cref{thm:pre_large_sieve} along with the ideas in \cite{ST} to prove \cref{thm:LFZDE,thm:subconvexity}.

\subsection*{Acknowledgements}

We thank Nickolas Andersen, Farrell Brumley, and Peter Humphries for helpful discussions.  We especially thank Kannan Soundararajan, who showed us how our idea for \cref{thm:large_sieve} leads to  \cref{prop:RSbound}, and the anonymous referees for their thorough suggestions and comments.  Work on this paper began while the authors were postdoctoral researchers at Stanford University.  Jesse Thorner was partially supported by an NSF Postdoctoral Fellowship.  Asif Zaman was partially supported by an NSERC fellowship.

\section{Properties of $L$-functions}
\label{sec:L-functions}


We recall some standard facts about $L$-functions arising from automorphic representations and their Rankin-Selberg convolutions; see \cite{Brumley,Michel} for convenient summaries.

\subsection{Standard $L$-functions}

Given $\pi\in\mathfrak{F}_n$, let $\widetilde{\pi}\in\mathfrak{F}_n$ be the contragredient representation.  Let $\kq_{\pi}$ be the arithmetic conductor of $\pi$.  We have $\kq_{\tilde{\pi}}=\kq_{\pi}$.  We express $\pi$ as a tensor product $\otimes_{v}\pi_{v}$ of smooth admissible representations of $\GL_n(F_{v})$, where $v$ varies over places of $F$.  For each nonarchimedean place $v$, there prime ideal $\kp$ corresponding to a non-archimedean place, we define $L(s,\pi_{\kp})$ in terms of the Satake parameters $A_{\pi}(\kp)=\{\alpha_{1,\pi}(\kp),\ldots,\alpha_{n,\pi}(\kp)\}$ by
\begin{equation}
	\label{eqn:Euler_p_single}
	L(s,\pi_{\kp})=\prod_{j=1}^{n}\Big(1-\frac{\alpha_{j,\pi}(\kp)}{\N\kp^{s}}\Big)^{-1}=\sum_{j=0}^{\infty}\frac{\lambda_{\pi}(\kp^j)}{\N\kp^{js}}.
\end{equation}
We have $\alpha_{j,\pi}(\kp)\neq0$ for all $j$ whenever $\kp\nmid\kq_{\pi}$, and it might be the case that $\alpha_{j,\pi}(\kp)=0$ for some $j$ when $\kp|\kq_{\pi}$.  The standard $L$-function $L(s,\pi)$ associated to $\pi$ is of the form
\[
L(s,\pi)=\prod_{\kp} L(s,\pi_{\kp})=\sum_{\kn}\frac{\lambda_{\pi}(\kn)}{\N\kn^s}.
\]
The Euler product and Dirichlet series converge absolutely when $\re(s)>1$.  We have the equality of sets $\{\alpha_{j,\widetilde{\pi}}(\kp)\}=\{\overline{ \alpha_{j,\pi}(\kp)}\}$.  At each archimedean place $v$ of $F$, there are $n$ Langlands parameters $\mu_{j,\pi}(v)\in\mathbb{C}$ from which we define
\[
L(s,\pi_{\infty}) = \prod_{ v|\infty}\prod_{j=1}^{n}\Gamma_{v}(s+\mu_{j,\pi}( v)),\qquad \Gamma_{v}(s):=\begin{cases}
	\pi^{-s/2}\Gamma(s/2)&\mbox{if $F_{ v}=\R$,}\\
	2(2\pi)^{-s}\Gamma(s)&\mbox{if $F_{ v}=\mathbb{C}$.}
\end{cases}
\]
We have the equality of sets $\{\mu_{j,\widetilde{\pi}}( v)\}=\{\overline{\mu_{j,\pi}( v)}\}$.  Let $d( v)=1$ if $F_{ v}=\R$ and $d( v)=2$ if $F_{ v}=\mathbb{C}$.  We define for $t\in\R$ the analytic conductor
\begin{equation}
\label{eqn:analytic_conductor_def}
C(\pi,t):=D_F^n \N\kq_{\pi}\prod_{ v|\infty}\prod_{j=1}^n(3+|it+\mu_{j,\pi}( v)|^{d( v)}),\qquad C(\pi):=C(\pi,0).
\end{equation}

Let $r_{\pi}=-\mathrm{ord}_{s=1} L(s,\pi)$, which is 1 if $\pi$ is trivial and 0 otherwise.  The completed $L$-function $\Lambda(s,\pi) = (s(s-1))^{r_{\pi}}(D_F^n \N\kq_{\pi})^{s/2}L(s,\pi)L(s,\pi_{\infty})$ is entire of order 1.  
Each pole of $L(s,\pi_{\infty})$ is a trivial zero of $L(s,\pi)$.  Since $\Lambda(s,\pi)$ is entire of order 1, a Hadamard factorization
\begin{equation}
\label{eqn:Hadamard}
\Lambda(s,\pi)=e^{a_{\pi}+b_{\pi}s}\prod_{\rho}\Big(1-\frac{s}{\rho}\Big)e^{s/\rho}
\end{equation}
exists, where $\rho$ varies over the nontrivial zeros of $L(s,\pi)$.  These zeros satisfy $0<\re(\rho)<1$.

There exists $\theta_n\in[0,\frac{1}{2}-\frac{1}{n^2+1}]$ such that for all pairs $(j,\kp)$ and $(j, v)$, we have  (see \cite{LRS,MS})
\begin{equation}
\label{eqn:LRS_finite}
	|\alpha_{j,\pi}(\kp)|\leq  \N\kp^{\theta_n}\quad\textup{and}\quad\re(\mu_{j,\pi}( v))\geq -\theta_n.
\end{equation}
The generalized Selberg conjecture and GRC assert that in \eqref{eqn:LRS_finite}, one may take $\theta_n=0$.

\subsection{Rankin--Selberg $L$-functions}
\label{subsec:RS}

Let $\pi\in\mathfrak{F}_n$ and $\pi'\in\mathfrak{F}_{n'}$.  Define
\begin{equation}
\label{eqn:RS_Dirichlet_series}
L(s,\pi_{\kp}\times\pi_{\kp}')=\prod_{j=1}^{n}\prod_{j'=1}^{n'}(1-\alpha_{j,j',\pi\times\pi'}(\kp) \N\kp^{-s})^{-1}=\sum_{j=0}^{\infty}\frac{\lambda_{\pi\times\pi'}(\kp^j)}{\N\kp^{js}}.
\end{equation}
for suitable complex numbers $\alpha_{j,j',\pi\times\pi'}(\kp)$.  If $\kp\nmid \kq_{\pi}\kq_{\pi'}$, then we have the equality of sets
\begin{equation}
\label{eqn:separate_dirichlet_coeffs}
\{\alpha_{j,j',\pi\times\pi'}(\kp)\}=\{\alpha_{j,\pi}(\kp)\alpha_{j',\pi'}(\kp)\}.
\end{equation}
See Brumley \cite[Appendix]{ST} for a description of $\alpha_{j,j',\pi\times\pi'}(\kp)$ when $\kp|\kq_{\pi}\kq_{\pi'}$. The Rankin-Selberg $L$-function $L(s,\pi\times\pi')$ associated to $\pi$ and $\pi'$ is of the form
\begin{equation}
\label{eqn:RS_ser}
L(s,\pi\times\pi')=\prod_{\kp}L(s,\pi_{\kp}\times\pi_{\kp}')=\sum_{\kn}\frac{\lambda_{\pi\times\pi'}(\kn)}{\N\kn^s}.
\end{equation}

At an archimedean place $v|\infty$, we define from the $n'n$ Langlands parameters $\mu_{j,j',\pi\times\pi'}( v)$ 
\[
L(s,\pi_{\infty}\times\pi_{\infty}') = \prod_{ v|\infty}\prod_{j=1}^{n}\prod_{j'=1}^{n'}\Gamma_{v}(s+\mu_{j,j',\pi\times\pi'}( v)).
\]
Let $\kq_{\pi\times\pi'}$ be the arithmetic conductor of $\pi\times\pi'$, and let
\begin{equation}
\label{eqn:order_res}
r_{\pi\times\pi'} = -\mathop{\mathrm{ord}}_{s=1}L(s,\pi\times\pi'),\qquad \kappa_{\pi\times\pi'}=\mathop{\mathrm{Res}}_{s=1}L(s,\pi\times\pi')\prod_{\kp|\kq_{\pi}\kq_{\pi'}}L(s,\pi_{\kp}\times\pi_{\kp}')^{-1}.
\end{equation}
By our normalization for the central characters of $\pi$ and $\pi'$, we have that $r_{\pi\times\pi'}=0$ and $\kappa_{\pi\times\pi'}=0$ if and only if $\pi'\neq \widetilde{\pi}$.  Otherwise, we have $r_{\pi\times\widetilde{\pi}}=1$, and the nonnegativity of $\lambda_{\pi\times\widetilde{\pi}}(\kn)$ when $(\kn,\kq_{\pi})=\mathcal{O}_F$ implies that $\kappa_{\pi \times \widetilde{\pi}} > 0$ (see \cref{cor:corcor} below).  The completed $L$-function $\Lambda(s,\pi\times\pi')=(s(s-1))^{r_{\pi\times\pi'}}(D_F^{n'n}\N\kq_{\pi\times\pi'})^{s/2}L(s,\pi\times\pi')L(s,\pi_{\infty}\times\pi_{\infty}')$ is entire of order 1, and hence possesses a Hadamard product.  Since \eqref{eqn:RS_ser} converges absolutely for $\re(s)>1$, it follows that $|\alpha_{j,j',\pi\times\pi'}(\kp)|<\N\kp$ and $\re(\mu_{j,j',\pi\times\pi'}(v))>-1$.

As with $L(s,\pi)$, we define
\[
C(\pi\times\pi',t):=D_F^{n'n}\N\kq_{\pi\times\pi'}\prod_{ v|\infty}\prod_{j=1}^n \prod_{j'=1}^{n'}(3+|it+\mu_{j,j',\pi\times\pi'}( v)|^{d( v)}),\quad C(\pi\times\pi'):=C(\pi\times\pi',0).
\]
The work of Bushnell and Henniart \cite[Theorem 1]{BH} and Brumley \cite[Lemma A.2]{Humphries} yields
\begin{equation}
\label{eqn:BH}
C(\pi\times\pi',t)\leq C(\pi\times\pi')(3+|t|)^{[F:\Q] n'n},\qquad C(\pi\times\pi')\leq e^{O(n'n)} C(\pi)^{n'}C(\pi')^{n}.
\end{equation}
For all $\epsilon>0$, Li's bound \cite[Theorem 2]{Li} and the Phragm{\'e}n--Lindel{\"o}f principle yield
\begin{equation}
	\label{eqn:Li}
	|(\sigma-1)^{r_{\pi\times\pi'}}L(\sigma+it,\pi\times\pi')|\ll_{n,[F:\Q],\epsilon}C(\pi\times\pi',t)^{\frac{1-\sigma}{2}+\epsilon},\qquad \sigma\leq 1.
\end{equation}
If $r_{\pi\times\pi'}=1$ and $\sigma=1$, then the left hand side of \eqref{eqn:Li} is viewed as a limit as $\sigma\to 1^+$.

\section{Rankin--Selberg combinatorics}
\label{sec:prelim_large_sieve}

A partition $\mu=(\mu_i)_{i=1}^{\infty}$ is a sequence of nonincreasing nonnegative integers $\mu_1\geq\mu_2\geq\cdots$ with finitely many nonzero entries.  For a partition $\mu$, let $\ell(\mu)$ be the number of $\mu_i\neq 0$, and let $|\mu|=\sum_{i=1}^{\infty} \mu_i$.  For a set $\{\alpha_{1},\ldots,\alpha_{n}\}\subseteq\mathbb{C}$ and a partition $\mu$ with $\ell(\mu)\leq n$, let
\[
s_{\mu}(\{\alpha_1,\ldots,\alpha_n\})=\det[(\alpha_{i}^{\mu(j)+n-j})_{ij}] / \det[(\alpha_{i}^{n-j})_{ij}]
\]
be the Schur polynomial associated to $\mu$.  If $|\mu|=0$, then $s_{\mu}(\{\alpha_1,\ldots,\alpha_n\})$ is identically one.  By convention, if $\ell(\mu)>n$, then $s_{\mu}(\{\alpha_1,\ldots,\alpha_n\})$ is identically zero.

Let $\pi\in\mathfrak{F}_n$, $\pi'\in\mathfrak{F}_{n'}$, and $\re(s)>1$.  Cauchy's identity \cite[(38.1)]{Bump_lie} implies that
\[
L(s,\pi_{\kp})=\prod_{j=1}^n\Big(1-\frac{\alpha_{j,\pi}(\kp)}{\N\kp^{s}}\Big)^{-1} = \sum_{k=0}^\infty \frac{s_{(k,0,\ldots)}(A_{\pi}(\kp))}{\N\kp^{ks}}
\]
and
\[
L(s,\pi_{\kp}\times\pi_{\kp}')=\prod_{j=1}^{n} \prod_{j'=1}^{n'}\Big(1-\frac{\alpha_{j,\pi}(\kp) \alpha_{j',\pi'}(\kp)}{\N\kp^{s}}\Big)^{-1} = \sum_{\mu}\frac{s_{\mu}(A_{\pi}(\kp))s_{\mu}(A_{\pi'}(\kp))}{\N\kp^{s|\mu|}},\quad \kp\nmid\kq_{\pi}\kq_{\pi'},
\]
where the sum ranges over all partitions.  The above identities, \eqref{eqn:Euler_p_single}, and \eqref{eqn:RS_Dirichlet_series} yield
\[
\lambda_{\pi}(\kp^k) = s_{(k,0,\ldots)}(A_{\pi}(\kp)),\qquad \lambda_{\pi\times\pi'}(\kp^k) = \sum_{ |\mu|=k}s_{\mu}(A_{\pi}(\kp))s_{\mu}(A_{\pi'}(\kp)).
\]
For an integral ideal $\kn$ with prime factorization $\kn=\prod_{\kp}\kp^{\mathrm{ord}_{\kp}(\kn)}$ (where $\mathrm{ord}_{\kp}(\kn)=0$ for all but finitely many prime ideals $\kp$), the multiplicativity of $\lambda_{\pi}(\kn)$ tells us that
\begin{equation}
	\label{eqn:hecke_schur}
	\lambda_{\pi}(\kn) = \prod_{\kp}\lambda_{\pi}(\kp^{\mathrm{ord}_{\kp}(\kn)}) = \prod_{\kp}s_{(\mathrm{ord}_{\kp}(\kn),0,\ldots)}(A_{\pi}(\kp)).
\end{equation}
Similarly, if $(\kn,\kq_{\pi}\kq_{\pi'})=\mathcal{O}_F$, then
\begin{align}
\label{n-prime2S}
\lambda_{\pi\times\pi'}(\kn) = \prod_{\kp}\lambda_{\pi\times\pi'}(\kp^{\mathrm{ord}_{\kp}(\kn)})=\sum_{(\mu_{\kp})_{\kp}\in\underline{\mu}[\kn]}\Big(\prod_{\kp}s_{\mu_{\kp}}(A_{\pi}(\kp))\Big)\Big(\prod_{\kp}s_{\mu_{\kp}}(A_{\pi'}(\kp))\Big).
\end{align}
Here, $(\mu_{\kp})_{\kp}$ denotes a sequence of partitions indexed by prime ideals and
\[
\underline{\mu}[\kn]:=\{(\mu_{\kp})_{\kp}\colon |\mu_{\kp}|=\mathrm{ord}_{\kp}(\kn)\textup{ for all $\kp$}\}.
\]
It is important to note that $s_{\mu_{\kp}}(A_{\widetilde{\pi}}(\kp))=\bar{s_{\mu_{\kp}}(A_{\pi}(\kp))}$.

We use \eqref{eqn:hecke_schur} and \eqref{n-prime2S} to prove a new inequality for Rankin--Selberg Dirichlet coefficients.
\begin{proposition}
\label{prop:RSbound}
	If $\pi\in\mathfrak{F}_n$ and $\pi'\in\mathfrak{F}_{n'}$ have conductors $\kq_{\pi},\kq_{\pi'}$ and $(\kn,\kq_{\pi}\kq_{\pi'})=\mathcal{O}_F$, then
	\begin{align*}
|\re(\lambda_{\pi\times\pi'}(\kn)-\lambda_{\pi}(\kn)\lambda_{\pi'}(\kn))|^2\leq (\lambda_{\pi\times\widetilde{\pi}}(\kn)-|\lambda_{\pi}(\kn)|^2)(\lambda_{\pi'\times\widetilde{\pi}'}(\kn)-|\lambda_{\pi'}(\kn)|^2).
	\end{align*}
\end{proposition}

\begin{proof}
	Let $\kn$ satisfy $(\kn,\kq_{\pi}\kq_{\pi'})=\mathcal{O}_F$, and let $x,y\in\R$.  By \eqref{eqn:hecke_schur},
\begin{align*}
|\lambda_{\pi}(\kn)x+\overline{\lambda_{\pi'}(\kn)}y|^2=\Big|x\prod_{\kp}s_{(\mathrm{ord}_{\kp}(\kn),0,\ldots)}(A_{\pi}(\kp))+y\overline{\prod_{\kp}s_{(\mathrm{ord}_{\kp}(\kn),0,\ldots)}(A_{\pi'}(\kp))}\Big|^2.
\end{align*}
Since $((\mathrm{ord}_{\kp}(\kn),0,\ldots))_{\kp}\in\underline{\mu}[\kn]$, it follows from nonnegativity that
\begin{equation}
\label{eqn:completing_the_square}
|\lambda_{\pi}(\kn)x+\overline{\lambda_{\pi'}(\kn)}y|^2\leq \sum_{(\mu_{\kp})_{\kp}\in\underline{\mu}[\kn]}\Big|x\prod_{\kp}s_{\mu_{\kp}}(A_{\pi}(\kp))+y\overline{\prod_{\kp}s_{\mu_{\kp}}(A_{\pi'}(\kp))}\Big|^2.
\end{equation}
Once we expand the squares on both sides of the inequality, we apply \eqref{n-prime2S} to deduce that
\[
|\lambda_{\pi}(\kn)|^2 x^2 + 2\re(\lambda_{\pi}(\kn)\lambda_{\pi'}(\kn))xy+|\lambda_{\pi'}(\kn)|^2 y^2\leq \lambda_{\pi\times\widetilde{\pi}}(\kn)x^2 + 2\re(\lambda_{\pi\times\pi'}(\kn))xy+\lambda_{\pi'\times\widetilde{\pi}'}(\kn)y^2.
\]
Hence the binary quadratic form $Q(x,y)\in\R[x,y]$ given by
\[
Q(x,y)=(\lambda_{\pi\times\widetilde{\pi}}(\kn)-|\lambda_{\pi}(\kn)|^2)x^2 + 2\re(\lambda_{\pi\times\pi'}(\kn)-\lambda_{\pi}(\kn)\lambda_{\pi'}(\kn))xy + (\lambda_{\pi'\times\widetilde{\pi}'}(\kn)-|\lambda_{\pi'}(\kn)|^2)y^2
\]
is positive-semidefinite and has a nonpositive discriminant.  The result follows.
\end{proof}
\begin{corollary}
\label{cor:corcor}
	If $\pi\in\mathfrak{F}_n$ has conductor $\kq_{\pi}$ and $(\kn,\kq_{\pi})=\mathcal{O}_F$, then $|\lambda_{\pi}(\kn)|^2\leq \lambda_{\pi\times\widetilde{\pi}}(\kn)$.
\end{corollary}
\begin{proof}
Since $Q(x,y)$ from the above proof is positive-semidefinite, we have $Q(1,0)\geq 0$.
\end{proof}
\begin{remark}
 While the conclusion of \cref{cor:corcor} is not new, the proof here appears to be new.	
\end{remark}

\section{Proof of \cref{thm:large_sieve}}
\label{sec:large_sieve}

\subsection{Overview of the strategy}
\label{sec:stratagem}
We quickly recall the strategy of Duke and Kowalski in \cite{DK} for proving \eqref{eqn:auto_large_sieve}.  Let $\alpha:\mathfrak{F}_n(Q)\to\mathbb{C}$ be an arbitrary function.  By the duality principle for bilinear forms, the bound \eqref{eqn:auto_large_sieve} is equivalent to the bound
\begin{equation}
\label{eqn:duke_dual}
\sum_{\N\kn\leq N}\Big|\sum_{\pi\in\mathfrak{F}_n(Q)}\lambda_{\pi}(\kn)\alpha(\pi)\Big|^2\ll_{n,[F:\Q],\epsilon}(NQ)^{\epsilon}(N+N^{\alpha}Q^{n(1-\alpha)}|\mathfrak{F}_n(Q)|)\sum_{\pi\in\mathfrak{F}_n(Q)}|\alpha(\pi)|^2.
\end{equation}
One expands the square on the left and swap the order of summation, arriving at
\[
\sum_{\pi\in\mathfrak{F}_n(Q)}\alpha(\pi)\overline{\alpha(\pi')}\sum_{\N\kn\leq N}\lambda_{\pi}(\kn)\overline{\lambda_{\pi'}(\kn)}=\sum_{\pi\in\mathfrak{F}_n(Q)}\alpha(\pi)\overline{\alpha(\pi')}\frac{1}{2\pi i}\int_{3-i\infty}^{3+i\infty}\Big(\sum_{\kn}\frac{\lambda_{\pi}(\kn)\lambda_{\widetilde{\pi}'}(\kn)}{\N\kn^s}\Big)\frac{x^s}{s}ds.
\]
The bound \eqref{eqn:duke_dual} now follows from the expected analytic continuation of the Dirichlet series in the integral, which we now describe.  When $\re(s)$ is large, there holds
\[
\sum_{\kn}\frac{\lambda_{\pi}(\kn)\lambda_{\widetilde{\pi}'}(\kn)}{\N\kn^s}=L(s,\pi\times\widetilde{\pi}')H(s,\pi,\pi'),\quad H(s,\pi,\pi')=\prod_{\kp}\frac{1}{L(s,\pi_{\kp}\times\tilde{\pi}_{\kp}')}\sum_{k=0}^{\infty}\frac{\lambda_{\pi}(\kp^k)\lambda_{\tilde{\pi}'}(\kp^k)}{\N\kp^{ks}}.
\]
(When $n=2$ and $\kq_{\pi}=\kq_{\pi'}=\mathcal{O}_F$, the identity $H(s,\pi,\pi')=\zeta_F(2s)^{-1}$ is classical.  Here, $\zeta_F(s)$ is the Dedekind zeta function of $F$).  Minor adjustments to the proof of \cite[Proposition 2]{DK} show that each Euler factor at $\kp$ for $H(s,\pi,\pi')$ is holomorphic when $\re(s) > \frac{1}{2}+2\theta_n$.  This analytically continues $H(s,\pi,\pi')$ to this region and produces the bound
\begin{equation}
\label{eqn:DKBound}
|H(s,\pi,\pi')|\ll_{n,[F:\Q],\epsilon}(C(\pi)C(\pi'))^{\epsilon}/(\re(s)-\tfrac{1}{2})^{c_n},\qquad \re(s) > \tfrac{1}{2}+2\theta_n
\end{equation}
for some constant $c_n>0$ depending at most on $n$.  This allows us to push the contour into the critical strip when $\theta_n\leq \frac{1}{4}-\delta$ for some fixed $\delta>0$.  Then \eqref{eqn:auto_large_sieve} holds with $\alpha = \frac{1}{2}+2\theta_n$.  Under GRC, we make take $\theta_n=0$.  When $n\leq 4$, Brumley \cite{Brumley_2} established these results (with $1-(n^2+1)^{-1}$ replacing $\frac{1}{2}+2\theta_n$) without recourse to unproven progress towards GRC.

The proof of \cref{prop:RSbound} epitomizes our strategy for \cref{thm:large_sieve}.  We begin our proof by rewriting \eqref{eqn:duke_dual} using \eqref{eqn:hecke_schur}, expressing $\lambda_{\pi}(\kn)$ as a product of Schur polynomials associated to a {\it particular} partition in $\underline{\mu}[\kn]$.  By nonnegativity, we may ``complete the sum'' by embedding \eqref{eqn:duke_dual} into a sum over {\it all} partitions in $\underline{\mu}[\kn]$, exactly as in \eqref{eqn:completing_the_square}.  Only then do we expand the square and swap the order of summation; now, instead of encountering sums of $\lambda_{\pi}(\kn)\lambda_{\tilde{\pi}'}(\kn)$ as Duke and Kowalski did, we encounter sums over $\lambda_{\pi\times\tilde{\pi}'}(\kn)$ because of \eqref{n-prime2S}.  This allows us to work directly with $L(s,\pi\times\tilde{\pi}')$ instead of $L(s,\pi\times\tilde{\pi}')H(s,\pi,\pi')$, which removes the source of unproven hypotheses in the work of Duke and Kowalski.  In preparation for \cref{thm:LFZDE}, we also prove a refinement of \cref{thm:large_sieve} when $a(\kn)$ is supported on prime ideals of large norm.  This involves a delicate synthesis of the large sieve with different Selberg sieve weights for each individual $\pi\in\mathfrak{F}_n(Q)$.

\subsection{An unconditional large sieve inequality}

We begin with a preliminary lemma.  Let $\phi$ be a smooth test function which is supported in a compact subset of $[-2,2]$, and let
\begin{equation}
\label{eqn:mellin}
\widehat{\phi}(s) = \int_{\R}\phi(y)e^{sy}dy
\end{equation}
be its Laplace transform.  Then $\widehat{\phi}(s)$ is an entire function of $s$, and for any integer $k\geq 0$,
\begin{equation}
\label{eqn:test_fcn_bounds}
\widehat{\phi}(s)\ll_{\phi,k}e^{2|\mathrm{Re}(s)|}|s|^{-k}.
\end{equation}
For any $x>0$, $T\geq 1$, and $c\in\R$, it follows from Laplace inversion that
\[
\phi(T\log x)=\frac{1}{2\pi i T}\int_{c-i\infty}^{c+i\infty}\widehat{\phi}(s/T)x^{-s}ds.
\]
\begin{lemma}
\label{lem:local_density}
Fix a test function $\phi$ as above.  Let $T,x\geq 1$.  Let $\pi,\pi'\in\mathfrak{F}_{n}(Q)$, let $\kd$ be a nonzero integral ideal, and define
\begin{equation}
	\label{eqn:local_density}
g_{\kd}(s,\pi\times\widetilde{\pi}')=\prod_{\kp| \kd }(1-L(s,\pi_{\kp}\times\widetilde{\pi}_{\kp}')^{-1}).
\end{equation}
If $\epsilon>0$, $\kd$ is squarefree, $(\kd,\kq_{\pi}\kq_{\pi'})=\mathcal{O}_F$, and $\kappa_{\pi\times\widetilde{\pi}'}$ is given by \eqref{eqn:order_res}, then
\[
\sum_{\substack{\kd|\kn \\ (\kn,\kq_{\pi}\kq_{\pi'})=\mathcal{O}_F}} \lambda_{\pi\times\widetilde{\pi}'}(\kn)\phi\Big(T\log\frac{\N\kn}{x}\Big)=g_{\kd}(1,\pi\times\widetilde{\pi}')x\frac{\widehat{\phi}(\frac{1}{T})}{T}\kappa_{\pi\times\widetilde{\pi}'}+O_{n,[F:\Q],\epsilon}(Q^{n^2+n+\epsilon} \N\kd^{n^2+\epsilon} T^{[F:\Q]n^2}).
\]
\end{lemma}
\begin{proof}
	We follow \cite[Lemma 6.1]{ST}.  The result is trivial for $x\leq e^2$.  For $x>e^2$, define $\sigma_x = (\log x)^{-1}$, so that once we push the contour to $\re(s)=\sigma_x$, the desired sum equals
	\[
	g_{\kd}(1,\pi\times\widetilde{\pi}')x\frac{\widehat{\phi}(\frac{1}{T})}{T}\kappa_{\pi\times\widetilde{\pi}'}+\frac{1}{2\pi i T}\int_{\sigma_x-i\infty}^{\sigma_x+i\infty}\frac{L(s,\pi\times\widetilde{\pi}')}{\prod_{\kp|\kq_{\pi}\kq_{\pi'}}L(s,\pi_{\kp}\times\widetilde{\pi}_{\kp}')}g_{\kd}(s,\pi\times\widetilde{\pi}')\widehat{\phi}(s/T)x^s ds.
	\]
It follows \eqref{eqn:Li} that
\[
|L(\sigma_x+it,\pi\times\widetilde{\pi}')|\ll_{n,[F:\Q],\epsilon}C(\pi\times\tilde{\pi}',t)^{\frac{1-\sigma_x}{2}+\frac{\epsilon}{4n}}\ll_{n,[F:\Q],\epsilon} C(\pi\times\tilde{\pi}',t)^{\frac{1}{2}+\frac{\epsilon}{4n}}.
\]
Since $C(\pi),C(\pi')\leq Q$, it follows from \eqref{eqn:BH} that
\[
C(\pi\times\tilde{\pi}',t)^{\frac{1}{2}+\frac{\epsilon}{4n}}\ll_{n,[F:\Q]} (C(\pi)^n C(\pi')^n (3+|t|)^{n^2[F:\Q]})^{\frac{1}{2}+\frac{\epsilon}{4n}} \ll_{n,[F:\Q]} (Q^{2n}(3+|t|)^{n^2[F:\Q]})^{\frac{1}{2}+\frac{\epsilon}{4n}}.
\]
Since $|\alpha_{j_1,j_2,\pi\times\widetilde{\pi}'}(\kp)|\leq \N\kp$ for all $j_1$, $j_2$, and $\kp$, it follows from \eqref{eqn:RS_Dirichlet_series} that for all $\epsilon>0$,
	\[
	\prod_{\kp|\kq_{\pi}\kq_{\pi'}}|L(\sigma_x+it,\pi_{\kp}\times\widetilde{\pi}_{\kp}')|^{-1}\leq \prod_{\kp|\kq_{\pi}\kq_{\pi'}}(1+\N\kp)^{n^2}\ll_{n,[F:\Q],\epsilon}Q^{n^2+\frac{\epsilon}{2}}.
	\]
	(Since $\N\kq_{\pi}\leq C(\pi)\leq Q$, a mild variant of the proof of \cite[Lemma 1.13]{Weiss} shows that $\#\{\kp|\kq_{\pi}\}\ll_{[F:\Q]}(\log  Q)/\log\log Q)$.)  Similarly, the bound $|g_{\kd}(\sigma_x+it,\pi\times\widetilde{\pi}')|\ll_{n,[F:\Q],\epsilon}\N\kd^{n^2+\epsilon}$ holds.  The integral is then
	\begin{align*}
	&\ll_{n,[F:\Q],\epsilon} \frac{1}{T}Q^{n^2+n+\epsilon} \N\kd^{n^2+\epsilon}\int_{-\infty}^{\infty}(2+|t|)^{[F:\Q]n^2}\Big|\widehat{\phi}\Big(\frac{\sigma_x+it}{T}\Big)\Big|dt\\
	&\ll_{n,[F:\Q],\epsilon} \frac{1}{T}Q^{n^2+n+\epsilon} \N\kd^{n^2+\epsilon}\int_{-\infty}^{\infty}(2+|t|)^{[F:\Q]n^2}\min\Big\{1,\frac{T^{[F:\Q]n^2+2}}{(2+|t|)^{[F:\Q]n^2+2}}\Big\}dt
	\end{align*}
by an application of \eqref{eqn:test_fcn_bounds}.  This is bounded as claimed.
\end{proof}

For integral ideals $\kq$ and $\kn$, we define $\delta_{(\kq,\kn)}$ to equal 1 if $(\kq,\kn)=\mathcal{O}_F$ and zero otherwise.

\begin{theorem}
\label{thm:pre_large_sieve}
	If $\epsilon>0$, $b(\kn)$ is a complex-valued function, and $Q,x\geq 1$, then
	\begin{equation}
	\label{eqn:first_part}
	\sum_{\pi\in\mathfrak{F}_{n}(Q)}\Big|\sum_{\substack{\N\kn\in(x,ex] \\ (\kn,\kq_{\pi})=\mathcal{O}_F}}\lambda_{\pi}(\kn)a(\kn)\Big|^2\ll_{n,[F:\Q],\epsilon}Q^{\epsilon}(x+Q^{n^2+n}|\mathfrak{F}_{n}(Q)|)\sum_{\N\kn\in(x,ex]}|a(\kn)|^2.
	\end{equation}
If $T\geq1$ and $z\ggg_{n,[F:\Q],\epsilon} Q^{2(n^2+n+\epsilon)}$, then
\begin{equation}
	\label{eqn:second_part}
	\begin{aligned}&\sum_{\pi\in\mathfrak{F}_{n}(Q)}\Big|\sum_{\substack{\N\kp\in(x,xe^{1/T}] \\ \N\kp>z}}\lambda_{\pi}(\kp)a(\kp)\Big|^2\\
	&\ll_{n,[F:\Q],\epsilon}\Big(\frac{x}{T\log z}+Q^{n^2+n+\epsilon}T^{[F:\Q]n^2}z^{2n^2+2+\epsilon}|\mathfrak{F}_n(Q)|\Big)\sum_{\substack{\N\kp\in(x,xe^{1/T}] \\ \N\kp>z}}|a(\kp)|^2.	
	\end{aligned}
	\end{equation}
\end{theorem}

\begin{proof}
We present a unified proof for both parts.  We begin by constructing Selberg sieve weights for each $\pi\in\mathfrak{F}_{n}(Q)$.  Define $P^{-}(\kn):=\min\{\N\kp\colon \kp|\kn\}$ with $P^{-}(\mathcal{O}_F)=\infty$.  Define
\[
g_{\pi}(\kd):=g_{\kd}(1,\pi\times\widetilde{\pi}),\qquad P_{\pi}(z):=\prod_{\substack{\N\kp<z\\ \kp\nmid\kq_{\pi},~g_{\pi}(\kp)\neq 0}}\kp,\qquad \mathcal{D}_{\pi}(z):=\{\kd\colon \N\kd\leq z,~\kd| P_{\pi}(z)\},
\]
where $g_{\kd}(1,\pi\times\widetilde{\pi})$ is given by \eqref{eqn:local_density}.  Let $\rho_{\pi}(\kd)$ be a real-valued function satisfying
\begin{equation}
\label{eqn:conditions}
\textup{$\rho_{\pi}(\mathcal{O}_F)=1$,}\qquad \textup{$\rho_{\pi}(\kd)=0$ unless $\kd\in\mathcal{D}_{\pi}(z)$,}\qquad  \textup{$|\rho_{\pi}(\kd)|\leq 1$ for all $\kd$.}
\end{equation}
By \eqref{eqn:conditions}, if $P^{-}(\kn)>z$, then the condition $\kd|\kn$ implies that either $\kd=\mathcal{O}_F$ or $\rho_{\pi}(\kd)=0$.

It suffices to consider $a$ such that $\sum_{\N\kn\in(x,xe^{1/T}]}|a(\kn)|^2=1$.  We will estimate the sum
\begin{equation}
	\label{eqn:ratio_ratio}
	\sum_{\pi\in\mathfrak{F}_n(Q)}\Big|\sum_{\substack{\N\kn\in(x,xe^{1/T}]}}\lambda_{\pi}(\kn)\delta_{(\kq_{\pi},\kn)}\Big[\sum_{\kd|(\kn,P_{\pi}(z))}\rho_{\pi}(\kd)\Big]a(\kn)\Big|^2.
\end{equation}
Define $\|\alpha\|_2 := (\sum_{\pi\in\mathfrak{F}_n(Q)}|\alpha(\pi)|^2)^{1/2}$.  By the duality principle for bilinear forms \cite[Section 7.3]{IK}, \eqref{eqn:ratio_ratio} is
\begin{equation}
\label{eqn:ratio_11}
\leq\sup_{\|\alpha\|_2=1}\sum_{\substack{\N\kn\in(x,xe^{1/T}] }}\Big|\sum_{\pi\in\mathfrak{F}_{n}(Q)}\lambda_{\pi}(\kn)\delta_{(\kq_{\pi},\kn)}\Big[\sum_{\kd|(\kn,P_{\pi}(z))}\rho_{\pi}(\kd)\Big]\alpha(\pi)\Big|^2.
\end{equation}
By \eqref{eqn:hecke_schur}, \eqref{eqn:ratio_11} equals the supremum over $\alpha$ with $\|\alpha\|_2=1$ of
\begin{equation}
\label{eqn:ratio_1}
\sum_{\substack{\N\kn\in(x,xe^{1/T}] }}\Big|\sum_{\pi\in\mathfrak{F}_{n}(Q)}\Big[\prod_{\kp}s_{(\mathrm{ord}_{\kp}(\kn),0,\ldots)}(A_{\pi}(\kp))\Big]\delta_{(\kq_{\pi},\kn)}\Big[\sum_{\kd|(\kn,P_{\pi}(z))}\rho_{\pi}(\kd)\Big]\alpha(\pi)\Big|^2.
\end{equation}
Since $((\mathrm{ord}_{\kp}(\kn),0,\ldots))_{\kp}\in\underline{\mu}[\kn]$, we bound \eqref{eqn:ratio_1} by embedding it into the ``completed sum''
\begin{equation}
\label{eqn:ratio_2}
\sum_{\substack{\N\kn\in(x,xe^{1/T}]}}~\sum_{(\mu_{\kp})_{\kp}\in\underline{\mu}[\kn]}\Big|\sum_{\pi\in\mathfrak{F}_{n}(Q)}\Big[\prod_{\kp}s_{\mu_{\kp}}(A_{\pi}(\kp))\Big]\delta_{(\kq_{\pi},\kn)}\Big[\sum_{\kd|(\kn,P_{\pi}(z))}\rho_{\pi}(\kd)\Big]\alpha(\pi)\Big|^2.
\end{equation}
Fix a nonnegative smooth function $\phi$ supported on a compact subset of $[-2,2]$ such that $\phi(t)\geq 1$ for $t\in[0,1]$.  Then \eqref{eqn:ratio_2} is
\begin{equation}
\label{eqn:ratio_3}
\leq\sum_{\kn}\sum_{(\mu_{\kp})_{\kp}\in\underline{\mu}[\kn]}\Big|\sum_{\pi\in\mathfrak{F}_{n}(Q)}\Big[\prod_{\kp}s_{\mu_{\kp}}(A_{\pi}(\kp))\Big]\delta_{(\kq_{\pi},\kn)}\Big[\sum_{\kd|(\kn,P_{\pi}(z))}\rho_{\pi}(\kd)\Big]\alpha(\pi)\Big|^2\phi\Big(T\log\frac{\N\kn}{x}\Big).
\end{equation}
We expand the square, swap the order of summation, apply \eqref{n-prime2S}, and see that \eqref{eqn:ratio_3} equals
\begin{align}
\label{eqn:RATIO_44}
\sum_{\pi,\pi'\in\mathfrak{F}_{n}(Q)}\alpha(\pi)&\overline{\alpha(\pi')}\Big[\sum_{(\kn,\kq_{\pi}\kq_{\pi'})=\mathcal{O}_F}\lambda_{\pi\times\widetilde{\pi}'}(\kn)\Big[\sum_{\kd|(\kn,P_{\pi}(z))}\rho_{\pi}(\kd)\Big]\Big[\sum_{\kd'|(\kn,P_{\pi'}(z))}\rho_{\pi'}(\kd')\Big]\phi\Big(T\log\frac{\N\kn}{x}\Big)\Big]\notag\\
&=\sum_{\pi,\pi'\in\mathfrak{F}_{n}(Q)}\alpha(\pi)\overline{\alpha(\pi')}\sum_{\substack{\kd\in\mathcal{D}_{\pi}(z) \\ \kd'\in\mathcal{D}_{\pi'}(z)}}\rho_{\pi}(\kd)\rho_{\pi'}(\kd')\Big[\sum_{\substack{[\kd,\kd']|\kn \\ (\kn,\kq_{\pi}\kq_{\pi'})=\mathcal{O}_F}}\lambda_{\pi\times\widetilde{\pi}'}(\kn)\phi\Big(T\log\frac{\N\kn}{x}\Big)\Big].
\end{align}
We apply \cref{lem:local_density} with $\kd$ replaced by $[\kd,\kd']$, so that \eqref{eqn:RATIO_44} equals
\begin{multline}
\label{eqn:triv_bound_11}
x\frac{\widehat{\phi}(\frac{1}{T})}{T}\sum_{\pi\in\mathfrak{F}_{n}(Q)}|\alpha(\pi)|^2 \kappa_{\pi\times\widetilde{\pi}}\sum_{\kd,\kd'\in\mathcal{D}_{\pi}(z) }\rho_{\pi}(\kd)\rho_{\pi}(\kd')g_{\pi}([\kd,\kd'])\\
+O_{n,[F:\Q],\epsilon}\Big(Q^{n^2+n+\epsilon}T^{[F:\Q]n^2}\sum_{\pi,\pi'\in\mathfrak{F}_{n}(Q)}|\alpha(\pi)\alpha(\pi')|\sum_{\substack{\kd\in\mathcal{D}_{\pi}(z) \\ \kd'\in\mathcal{D}_{\pi'}(z)}}|\rho_{\pi}(\kd)\rho_{\pi'}(\kd')|\cdot\N[\kd,\kd']^{n^2+\epsilon}\Big).
\end{multline}
(The ``off-diagonal contribution'' arising from the pairs $\pi\neq\pi'$ resides in the error term because $\kappa_{\pi\times\widetilde{\pi}'}=0$ if and only if $\pi\neq\pi'$.)  By \cite[Lemma 1.12a]{Weiss}, we have the bound $\sum_{\N\kd\leq z}1\ll_{[F:\Q],\epsilon} z^{1+\epsilon}$ for all $z\geq 1$ and all $\epsilon>0$.  This bound, along with \eqref{eqn:conditions} and the inequality of arithmetic and geometric means, implies that \eqref{eqn:triv_bound_11} equals (recall $\|\alpha\|_2=1$)
\begin{multline}
\label{eqn:almost_there}
x\frac{\widehat{\phi}(\frac{1}{T})}{T}\sum_{\pi\in\mathfrak{F}_{n}(Q)}|\alpha(\pi)|^2\kappa_{\pi\times\widetilde{\pi}}\sum_{\kd,\kd'\in\mathcal{D}_{\pi}(z)}\rho_{\pi}(\kd)\rho_{\pi}(\kd')g_{\pi}([\kd,\kd'])\\
+O_{n,[F:\Q],\epsilon}\big(Q^{n^2+n+\epsilon}T^{[F:\Q]n^2}z^{2n^2+2+\epsilon}|\mathfrak{F}_n(Q)|\big).
\end{multline}

Proceeding as in the formulation of the Selberg sieve in \cite[Theorem 7.1]{FI}, we find that for each $\pi\in\mathfrak{F}_{n}(Q)$, there exists a choice of $\rho_{\pi}(\kd)$ satisfying \eqref{eqn:conditions} such that
\begin{align*}
\sum_{\kd,\kd'\in\mathcal{D}_{\pi}(z)}\rho_{\pi}(\kd)\rho_{\pi}(\kd')g_{\pi}([\kd,\kd'])&=\Big(\sum_{\substack{\N\kd\leq z \\ \kd| P_{\pi}(z)}}\prod_{\kp| \kd}\frac{g_{\pi}(\kp)}{1-g_{\pi}(\kp)}\Big)^{-1}\leq\Big(\sum_{\substack{\N\kn\leq z \\ (\kn,\kq_{\pi})=\mathcal{O}_F \\ \textup{$\kn$ squarefree}}}\prod_{\kp|\kn}\sum_{j=1}^{\infty}\frac{\lambda_{\pi\times\widetilde{\pi}}(\kp^j)}{\N\kp^j}\Big)^{-1}.
\end{align*}
Hence \eqref{eqn:almost_there} is
\begin{equation}
\label{eqn:almost_there2}
\leq x\frac{\widehat{\phi}(\frac{1}{T})}{T}\sum_{\pi\in\mathfrak{F}_{n}(Q)}\frac{|\alpha(\pi)|^2\kappa_{\pi\times\widetilde{\pi}}}{\displaystyle\sum_{\substack{\N\kn\leq z \\ (\kn,\kq_{\pi})=\mathcal{O}_F}}\lambda_{\pi\times\widetilde{\pi}}(\kn)}+O_{n,[F:\Q],\epsilon}\big(Q^{n^2+n+\epsilon}T^{[F:\Q]n^2}z^{2n^2+2+\epsilon}|\mathfrak{F}_n(Q)|\big).
\end{equation}
Since $\phi$ is fixed, \eqref{eqn:mellin} implies that $\widehat{\phi}(\frac{1}{T})\ll 1$.  Therefore, since $\|\alpha\|_2=1$, we conclude that
\begin{multline}
\label{eqn:fixed_bound}
\sum_{\pi\in\mathfrak{F}_n(Q)}\Big|\sum_{\substack{\N\kn\in(x,xe^{1/T}]}}\lambda_{\pi}(\kn)\delta_{(\kq_{\pi},\kn)}\Big[\sum_{\kd|(\kn,P_{\pi}(z))}\rho_{\pi}(\kd)\Big]a(\kn)\Big|^2\\
\ll_{n,[F:\Q],\epsilon}\frac{x}{T}\max_{\pi\in\mathfrak{F}_{n}(Q)}\kappa_{\pi\times\widetilde{\pi}}\Big(\sum_{\substack{\N\kn\leq z \\ (\kn,\kq_{\pi})=\mathcal{O}_F}}\frac{\lambda_{\pi\times\widetilde{\pi}}(\kn)}{\N\kn}\Big)^{-1}+Q^{n^2+n+\epsilon}T^{[F:\Q]n^2}z^{2n^2+2+\epsilon}|\mathfrak{F}_{n}(Q)|.
\end{multline}

To prove \eqref{eqn:first_part}, recall that $|\alpha_{j,j',\pi\times\widetilde{\pi}}(\kp)|\leq \N\kp$ for all $\kp$, hence
\[
\prod_{\kp|\kq_{\pi}}|L(1,\pi_{\kp}\times\widetilde{\pi}_{\kp})|^{-1}=\prod_{\kp|\kq_{\pi}}\prod_{j,j'=1}^n\Big|1-\frac{\alpha_{j,j',\pi\times\tilde{\pi}}(\kp)}{\N\kp}\Big|\leq \prod_{\kp|\kq_{\pi}}\prod_{j,j'=1}^n 2=\prod_{\kp|\kq_{\pi}}2^{n^2}.
\]
As mentioned earlier, we have $\#\{\kp|\kq_{\pi}\}\ll_{[F:\Q]}(\log  Q)/\log\log Q$.  We combine these estimates with \eqref{eqn:BH} and \eqref{eqn:Li} to find that $\kappa_{\pi\times\widetilde{\pi}}\ll_{n,[F:\Q],\epsilon}Q^{\epsilon}$ for all $\epsilon>0$.  Also, for all $\kn$, we have $\sum_{\kd|(\kn,P_{\pi}(1))}\rho_{\pi}(\kd) = 1$.  Thus \eqref{eqn:first_part} follows from \eqref{eqn:fixed_bound} with $T=z=1$.

We now prove \eqref{eqn:second_part}.  Note that if $\kp$ is a prime ideal with $\N\kp >z$, then $\sum_{\kd|(\kp,P_{\pi}(z))}\rho_{\pi}(\kd) = 1$.  Therefore, if we choose $a$  in  \eqref{eqn:fixed_bound} so that $a(\kn)=0$ unless $\kn$ is prime and $\N\kn>z$, then it suffices to prove that
	\begin{equation}
	\label{eqn:zrange1}
\sum_{\substack{\N\kn\leq z \\ (\kn,\kq_{\pi})=\mathcal{O}_F}}\frac{\lambda_{\pi\times\widetilde{\pi}}(\kn)}{\N\kn}\geq \frac{1+\kappa_{\pi\times\widetilde{\pi}}\log z}{3},\qquad z\ggg_{n,[F:\Q],\epsilon} Q^{2(n^2+n+\epsilon)}.
	\end{equation}
Fix $\epsilon\in(0,1/4)$.  Fix a smooth nonnegative function $\phi_1$ which is compactly supported in $[0,1]$, with $\phi_1(t)=1$ for $t\in[\epsilon,1-\epsilon]$ and $\phi_1(t)\leq 1$ for $t\in[0,1]$.  If $y\geq 1$ and $\N\kn\in(y,ey]$, then $\N\kn^{-1}\geq (ey)^{-1}\phi_1(\log\frac{\N\kn}{y})$.  By \cref{lem:local_density} with $\kd=\mathcal{O}_F$ and $T=1$, we have
\begin{align}
	\label{eqn:lowerboundz}
	\sum_{\substack{\N\kn\in(y,ey]  \\ (\kn,\kq_{\pi})=\mathcal{O}_F}}\frac{\lambda_{\pi\times\widetilde{\pi}}(\kn)}{\N\kn}\geq \sum_{(\kn,\kq_{\pi})=\mathcal{O}_F}\frac{\lambda_{\pi\times\widetilde{\pi}}(\kn)}{ey}\phi_1\Big(\log\frac{\N\kn}{y}\Big)=\frac{\widehat{\phi}_1(1)}{e}\kappa_{\pi\times\widetilde{\pi}}+O_{n,[F:\Q],\epsilon}\Big(\frac{Q^{n^2+n+\epsilon}}{y}\Big).
	\end{align}
	If $\epsilon$ is sufficiently small, then $\widehat{\phi}_1(1)=e-1+O(\epsilon)$.  We dyadically subdivide $[\sqrt{z},z]$ and use \eqref{eqn:lowerboundz} to obtain (recall that $\kappa_{\pi\times\tilde{\pi}}\ll_{\epsilon,n,[F:\Q]}Q^{\epsilon}$ by \eqref{eqn:BH} and \eqref{eqn:Li})
	\begin{align*}
		\sum_{\substack{\N\kn\leq z \\ (\kn,\kq_{\pi})=\mathcal{O}_F}}\frac{\lambda_{\pi\times\widetilde{\pi}}(\kn)}{\N\kn}\geq 1+\sum_{\substack{\sqrt{z}\leq \N\kn\leq z \\ (\kn,\kq_{\pi})=\mathcal{O}_F}}\frac{\lambda_{\pi\times\widetilde{\pi}}(\kn)}{\N\kn}\geq 1+\frac{\kappa_{\pi\times\widetilde{\pi}}}{3}\log z + O_{n,[F:\Q],\epsilon}\Big(\frac{Q^{n^2+n+\epsilon}}{\sqrt{z}}\Big).
	\end{align*}
	Once $z\ggg_{n,[F:\Q],\epsilon} Q^{2(n^2+n+\epsilon)}$, we achieve \eqref{eqn:zrange1}.
\end{proof}

\begin{proof}[Proof of \cref{thm:large_sieve}]
Dyadically decompose $[1,N]$ and sum the contributions from each subinterval using \eqref{eqn:first_part} and the Cauchy--Schwarz inequality.
\end{proof}

\begin{remark}
Even with the description of $L(s,\pi_{\kp}\times\pi_{\kp}')$ when $\kp|\kq_{\pi}\kq_{\pi'}$ given by Brumley \cite[Appendix]{ST}, there appears to be no variant of \eqref{n-prime2S} when $(\kn,\kq_{\pi}\kq_{\pi'})\neq\mathcal{O}_F$ which holds with enough uniformity in $\pi$ and $\pi'$ to allow us to remove the $(\kn,\kq_{\pi})=\mathcal{O}_F$ condition in \cref{thm:large_sieve}.  Such a removal would eliminate the need work with the unramified Rankin--Selberg $L$-functions, and the $Q^{n^2+n}$ term would improve to $Q^{n}$, even without GRC.
\end{remark}

\subsection{Mean value estimates for Dirichlet polynomials}
\label{subsec:MVT}

The corollary below, whose proof relies on \cref{thm:pre_large_sieve}, serves as a key component in our proof of \cref{thm:LFZDE}.

\begin{corollary}
\label{cor:MVT_primes}
	Let $Q,T\geq 1$.  If $y\ggg (Q T^{[F:\Q]})^{60n^4}$ and $u\in[y,y^{12000}]$, then
	\[
	\sum_{\pi\in\mathfrak{F}_{n}(Q)}\int_{-T}^{T}\Big|\sum_{y<\N\kp\leq u}\frac{\lambda_{\pi}(\kp)\log\N\kp}{\N\kp^{1+it}}\Big|^2 dt\ll_{n,[F:\Q]} \log u.
	\] 
\end{corollary}
\begin{proof}
	A formal generalization of \cite[Theorem 1]{Gallagher} to number fields tells us that for $b(\kn)$ a complex-valued function supported on the integral ideals of $\mathcal{O}_F$ such that $\sum_{\kn}|b(\kn)|<\infty$,
	\[
	\int_{-T}^{T}\Big|\sum_{\kn}b(\kn)\N\kn^{-it}\Big|^2 dt\ll T^2\int_0^{\infty}\Big|\sum_{\N\kn\in(x,xe^{1/T}]}b(\kn)\Big|^2\frac{dx}{x}.
	\]
	Therefore, if $\sum_{\N\kp>z}|\lambda_{\pi}(\kp)a(\kp)|\N\kp<\infty$, then
	\begin{equation}
	\label{eqn:ticks}
		\sum_{\pi\in\mathfrak{F}_{n}(Q)}\int_{-T}^{T}\Big|\sum_{\N\kp>z}\lambda_{\pi}(\kp)a(\kp)\N\kp^{-it}\Big|^2 dt\ll T^2\int_0^{\infty}\sum_{\pi\in\mathfrak{F}_{n}(Q)}\Big|\sum_{\substack{\N\kp\in(x,xe^{1/T}] \\ \N\kp>z}}\lambda_{\pi}(\kp)a(\kp)\Big|^2\frac{dx}{x}.
	\end{equation}
	By \eqref{eqn:poly_upper} with $\epsilon=\frac{1}{2}$, we have $Q^{n^2+n+\epsilon}|\mathfrak{F}_n(Q)|\ll_{n,[F:\Q]}Q^{5n^2}$.  Let $y\ggg (QT^{[F:\Q]})^{60n^4}$ and  $z = y^{\frac{6}{60n^2}}$, which ensures that $z\ggg_{n,[F:\Q]} Q^{2(n^2+n+1)}$.  Let $u\in[y,y^{12000}]$, and define $a(\kp)$ by $(\log\N\kp)/\N\kp$ if $\N\kp\in[y,u]$ and zero otherwise.   An application of \eqref{eqn:second_part} shows that \eqref{eqn:ticks} is
	\begin{equation}
	\label{eqn:leeches}
		\ll_{n,[F:\Q]} \sum_{\N\kp>z}|a(\kp)|^2 \N\kp\Big(\frac{1}{\log z}+\frac{Q^{5n^2}T^{[F:\Q]n^2+1}z^{2n^2+3}}{\N\kp}\Big)\ll_{n,[F:\Q]} \sum_{y\leq \N\kp\leq u}\frac{(\log \N\kp)^2}{\N\kp\log y}.
	\end{equation}
This is $\ll \log u$ by \cite[Lemma 1.11b]{Weiss}, partial summation, and the range of $u$.
\end{proof}

\section{Proof of \cref{thm:LFZDE,thm:subconvexity}}
\label{sec:ZDE}

Our approach to \cref{thm:LFZDE} closely follows the approach in \cite{ST}, which handles the case of a single $\pi\in\mathfrak{F}_{n}$ over $F=\Q$.  While we only require minor modifications, our treatment is self-contained apart from a few standard calculations which do not directly pertain to our application of \cref{cor:MVT_primes} in the proof.  For $\pi\in\mathfrak{F}_n(Q)$, we define the coefficients $\Lambda_{\pi}(\kn)$ by
\[
\sum_{\kn}\frac{\Lambda_{\pi}(\kn)}{\N\kn^s}=-\frac{L'}{L}(s,\pi) = \sum_{k=1}^{\infty}\sum_{\kp}\frac{\sum_{j=1}^{n}\alpha_{j,\pi}(\kp)^k\log\N\kp}{\N\kp^{ks}},\qquad \re(s)>1.
\]
We similarly define $\Lambda_{\pi\times\tilde{\pi}}(\kn)$ as the $\kn$-th Dirichlet coefficient of $-\frac{L'}{L}(s,\pi\times\tilde{\pi})$.  We have $\Lambda_{\pi\times\tilde{\pi}}(\kn)\geq 0$ \cite[(A.8) and (A.11)]{ST} and $2|\Lambda_{\pi}(\kn)|\leq \Lambda_{\pi\times\tilde{\pi}}(\kn)+1$ \cite[Prop. A.1]{ST}.

\subsection{Preliminary lemmas}

Suppose that $r_{\pi}=0$, in which case $L(s,\pi)$ is entire.  Taking logarithmic derivatives of both sides of \eqref{eqn:Hadamard}, we see that
\begin{equation}
\label{eqn:3.2}
\sum_{\rho}\Big(\frac{1}{s-\rho}+\frac{1}{\rho}\Big)+b_{\pi}=\frac{L'}{L}(s,\pi)+\frac{\log(D_F^{n}\kq_{\pi})}{2}+\frac{L'}{L}(s,\pi_{\infty}).
\end{equation}
Since $\re(b_{\pi})=-\sum_{\rho}\re(\rho^{-1})$ \cite[Proposition 5.7(3)]{IK},  we have
\begin{equation}
\label{eqn:3.3}
	\sum_{\rho}\re\Big(\frac{1}{s-\rho}\Big)=\re\Big(\frac{L'}{L}(s,\pi)+\frac{L'}{L}(s,\pi_{\infty})\Big)+\frac{\log(D_F^{n}\kq_{\pi})}{2}.
\end{equation}

\begin{lemma}
	\label{lem:mertens}
	If $\eta>0$, then
	\[
	\sum_{\kn}\frac{|\Lambda_{\pi}(\kn)|}{\N\kn^{1+\eta}}\leq\frac{1}{\eta}+n\log C(\pi) + O_{n,[F:\Q]}(1).
	\]
\end{lemma}
\begin{proof}
Since $\Lambda_{\pi\times\tilde{\pi}}(\kn)\geq 0$, by the above discussion, we find that
\[
\sum_{\kn}\frac{|\Lambda_{\pi}(\kn)|}{\N\kn^{1+\eta}}\leq \frac{1}{2}\sum_{\kn}\frac{\Lambda_{\pi\times\tilde{\pi}}(\kn)+1}{\N\kn^{1+\eta}}=-\frac{1}{2}\frac{L'}{L}(1+\eta,\pi\times\tilde{\pi})-\frac{1}{2}\frac{\zeta_F'}{\zeta_F}(1+\eta).
\]
This follows from standard manipulations of the Hadamard product for $L(s,\pi\times\tilde{\pi})$ and \eqref{eqn:BH}.  See \cite[Lemma 2.3]{ST} for details.
\end{proof}

\begin{lemma}
\label{lem:3.1}
	If $0<\eta\leq 1$ and $t\in\R$, then
	\[
	\sum_{\rho}\frac{1+\eta-\beta}{|1+\eta+it-\rho|^2}\leq 2n\log Q+n[F:\Q]\log(2+|t|)+\frac{2n}{\eta}+O_{n,[F:\Q]}(1).
	\]
	and $\#\{\rho\colon |\rho-(1+it)|\leq\eta\}\leq 10\eta n\log Q+5\eta n[F:\Q]\log(2+|t|)+O_{n,[F:\Q]}(1)$.
\end{lemma}
\begin{proof}[Sketch of proof]
These follow from elementary manipulations of \eqref{eqn:3.3} that use \eqref{lem:mertens} to handle $\re(\frac{L'}{L}(s,\pi))$.  See \cite[Lemma 3.1]{ST} for details.
\end{proof}

\subsection{Detecting zeros}
\label{subsec:detection}

Let $k\geq 1$ be an integer, and let
\begin{equation}
\label{eqn:eta_range}
(\log QT^{[F:\Q]})^{-1}<\eta\leq (200n^3)^{-1}.
\end{equation}
Let $s=1+\eta+i\tau$.  If $L(s,\pi)$ is entire, then differentiating \eqref{eqn:3.2} $k$ times, we find that
\begin{align*}
	\Big(\frac{L'}{L}(s,\pi)\Big)^{(k)}+\Big(\frac{L'}{L}(s,\pi_{\infty})\Big)^{(k)}=(-1)^k k!\sum_{\rho}\frac{1}{(s-\rho)^{k+1}}.
\end{align*}
Using the Hadamard formulation of $\Gamma(s)$ and \eqref{eqn:LRS_finite} to handle the contribution from $L(s,\pi_{\infty})$ and \cref{lem:3.1} to handle the contribution from zeros $\rho$ with $|s-\rho|>200\eta$, we find that
\begin{equation}
	\label{eqn:4.2}
	\Big|\frac{(-1)^k}{k!}\Big(\frac{L'}{L}(s,\pi)\Big)^{(k)}-\sum_{|s-\rho|\leq 200\eta}\frac{1}{(s-\rho)^{k+1}}\Big|\ll \frac{n\log(QT^{[F:\Q]})}{(200\eta)^k}.
\end{equation}

If $L(z,\pi)$ has a zero $\rho_0$ satisfying $|s-\rho|\leq 200\eta$ (with $s=1+\eta+i\tau$), then we will produce upper {\it and} lower bounds for the high derivatives of $\frac{L'}{L}(s,\pi)$.  This leads to a criterion by which we can detect zeros near $\re(s)=1$ (see \eqref{eqn:zerodetect}).  The interplay between the upper and lower bounds will produce the desired zero density estimate.  If the sum over zeros in \eqref{eqn:4.2} is not empty, then we can appeal to the following result of S{\'o}s and Tur{\'a}n \cite{Turan}.
\begin{lemma}
\label{lem:turan}
	Let $z_1,\ldots,z_{\nu}\in\mathbb{C}$.  If $K\geq \nu$, then there exists an integer $k\in[K,2K]$ such that $|z_1^k+\cdots+z_{\nu}^k|\geq(|z_1|/50)^k$.
\end{lemma}

We begin with the lower bound.
\begin{lemma}
	\label{lem:4.2}
	Let $\tau\in\R$ satisfy $|\tau|\leq T$, and let $\eta$ satisfy \eqref{eqn:eta_range}.  If $L(z,\pi)$ has a zero $\rho_0$ satisfying $|\rho_0-(1+it)|\leq\eta$ and $K>\lceil 2000\eta n \log(QT^{[F:\Q]})+O_{n,[F:\Q]}(1)\rceil$ with a sufficiently large implied constant, then for some integer $k\in[K,2K]$, one has (recall $s=1+\eta+i\tau$)
	\[
	\Big|\sum_{\substack{\rho \\ |s-\rho|\leq200\eta}}\frac{1}{(s-\rho)^{k+1}}\Big|\geq\frac{1}{(100\eta)^{k+1}}.
	\]
\end{lemma}
\begin{proof}
	By \cref{lem:3.1}, there are $\leq 2000\eta n \log(QT^{[F:\Q]})+O_{n,[F:\Q]}(1)$ zeros satisfying $|s-\rho|\leq 200\eta$.  If $z_1=(s-\rho_0)^{-1}$, then $|z_1|\geq 1/(2\eta)$.  Thus we may apply \cref{lem:turan}.
\end{proof}

Suppose now that $K=w\eta n \log(QT^{[F:\Q]})+O_{n,[F:\Q]}(1)$, where $w>2000$ is an absolute constant (to be determined shortly) and the implied constant is sufficiently large.  If $|\tau|\leq T$ and $L(z,\pi)$ has a zero $\rho_0$ satisfying $|1+i\tau-\rho_0|\leq \eta$, then we can combine \eqref{eqn:4.2} with \cref{lem:4.2} and determine that for some $k\in [K,2K]$,
\begin{equation}
\label{high_deriv_lower}
\Big|\frac{\eta^{k+1}}{k!}\Big(\frac{L'}{L}(s,\pi)\Big)^{(k)}\Big|\geq\Big(\frac{1}{100}\Big)^{k+1}\Big(1-O_{n,[F:\Q]}\Big(\frac{\eta \log(C(\pi)T^{[F:\Q]})}{2^k}\Big)\Big)\geq\frac{1}{2(100)^{k+1}}.
\end{equation}
We proceed to the upper bound.

\begin{lemma}
\label{lem:prime_power_contribution}
	Let $L(s,\pi)$ be entire.  Let $T\geq 1$, let $\tau\in\R$ satisfy $|\tau|\leq T$, and let $\eta$ satisfy \eqref{eqn:eta_range}.  Let $K\geq 1$ and $k\in[K,2K]$ be integers, and put $N_0=\exp(K/(300\eta))$ and $N_1=\exp(40K/\eta)$.  Set $s=1+\eta+i\tau$.  One has that
	\[
	\Big|\frac{\eta^{k+1}}{k!}\Big(\frac{L'}{L}(s,\pi)\Big)^{(k)}\Big|\leq \eta^2\int_{N_0}^{N_1}\Big|\sum_{N_0\leq \N\kp\leq u}\frac{\lambda_{\pi}(\kp)\log\N\kp}{\N\kp^{1+i\tau}}\Big|\frac{du}{u}+O_{n,[F:\Q]}\Big(\frac{\eta \log(QT^{[F:\Q]})}{(110)^k}\Big).
	\]
	\end{lemma}

	\begin{proof}
Since $\eta>0$, we have
\[
\Big|\frac{\eta^{k+1}}{k!}\Big(\frac{L'}{L}(s,\pi)\Big)^{(k)}\Big|=\eta\Big|\sum_{\kn}\frac{\Lambda_{\pi}(\kn)}{\N\kn^{1+\eta+i\tau}}\frac{(\eta\log\N\kn)^k}{k!}\Big|.
\]
It is straightforward to check that $(\eta\log\N\kn)^k / (\N\kn^{-\eta}k!)\leq \N\kn^{-\eta/2}(110)^{-k}$ when $\N\kn\notin[N_0,N_1]$ (see \cite[Lemma 4.3]{ST}).  Since $\sum_{\kn}|\Lambda_{\pi}(\kn)|\N\kn^{-1-\eta}\ll_{n,[F:\Q]}\log(QT^{[F:\Q]})$ by \cref{lem:mertens}, we have
\begin{equation}
\label{eqn:prime_power_contribution}
\Big|\sum_{\kn\notin[N_0,N_1]}\frac{\Lambda_{\pi}(\kn)}{\N\kn^{1+\eta+i\tau}}\frac{(\eta\log\N\kn)^k}{k!}\Big|\ll \frac{1}{(110)^k}\sum_{\kn}\frac{|\Lambda_{\pi}(\kn)|}{\N\kn^{1+\eta/2}}\ll_{n,[F:\Q]} \frac{\log(QT^{[F:\Q]})}{(110)^k}.
\end{equation}

Let $n\geq 2$.  Consider the composite $\kn$ with $\N\kn\in[N_0,N_1]$.  Since $(\log u)^k\leq k!u$ for $k,u\geq 1$,
\[
\frac{(\eta\log \N\kn)^k}{k!}=((n^2+1)\eta)^k\frac{(\log \N\kn^{\frac{1}{n^2+1}})^k}{k!}\leq ((n^2+1)\eta)^k \N\kn^{\frac{1}{n^2+1}}\leq\frac{\N\kn^{\frac{1}{n^2+1}}}{(110)^k}.
\]
The above estimate and the bound \eqref{eqn:LRS_finite} imply that
\begin{align*}
(110)^k\Big|\sum_{\substack{\N\kn\in[N_0,N_1] \\ \textup{$\kn$ composite}}}\frac{\Lambda_{\pi}(\kn)}{\N\kn^{1+\eta+i\tau}}\frac{(\eta \log \N\kn)^k}{k!}\Big|\ll \sum_{\substack{\N\kn\in[N_0,N_1] \\ \textup{$\kn$ composite}}}\frac{|\Lambda_{\pi}(\kn)|}{\N\kn^{1-\frac{1}{n^2+1}+\eta}}\ll_{n,[F:\Q]} \sum_{\kp}\sum_{r\geq 2}\frac{\log \N\kp}{\N\kp^{r(\frac{1}{2}+\eta)}}.
\end{align*}
The final sum is $\ll n(110)^{-k}\log(QT^{[F:\Q]})$ by \cite[Lemma 1.11b]{Weiss} and \eqref{eqn:eta_range}, so
\begin{equation}
\label{eqn:prime_power_contribution'}
	\Big|\sum_{\substack{\N\kn\in[N_0,N_1] \\ \textup{$\kn$ composite}}}\frac{\Lambda_{\pi}(\kn)}{\N\kn^{1+\eta+i\tau}}\frac{(\eta \log \N\kn)^k}{k!}\Big|\ll_{n,[F:\Q]} \frac{\log(QT^{[F:\Q]})}{(110)^k}.
\end{equation}
Stronger results hold when $n=1$, in which case GRC holds (see \cite[Section 5]{TZ1}).

Let $j_k(u)=e^{-u}u^k/k!$, so that $j_k(\eta\log\N\kn)=(\eta\log\N\kn)^k/(\N\kn^{\eta}k!)$.  Since $|\frac{d}{du}j_k(\eta\log u)|=|j_{k-1}(\eta\log u)-j_k(\eta\log u)|(\eta/u)\leq\eta/u$, it follows by partial summation and \cref{lem:mertens} that
\begin{align}
\label{eqn:prime_power_contribution_3}
\Big|\sum_{\N\kp\in[N_0,N_1] }\frac{\Lambda_{\pi}(\kp)}{\N\kp^{1+\eta+i\tau}}\frac{(\eta \log \N\kp)^k}{k!}\Big|&\leq \eta\int_{N_0}^{N_1}\Big|\sum_{N_0\leq\N\kp\leq u}\frac{\Lambda_{\pi}(\kp)}{\N\kp^{1+i\tau}}\Big|\frac{du}{u}+O\Big(\frac{1}{(110)^k}\sum_{\kn}\frac{|\Lambda_{\pi}(\kn)|}{\N\kn^{1+\eta/2}}\Big)\notag\\
&=\eta\int_{N_0}^{N_1}\Big|\sum_{N_0\leq\N\kp\leq u}\frac{\Lambda_{\pi}(\kp)}{\N\kp^{1+i\tau}}\Big|\frac{du}{u} +O_{n,[F:\Q]}\Big(\frac{\log(QT^{[F:\Q]})}{(110)^k}\Big).
\end{align}
The lemma follows from \eqref{eqn:prime_power_contribution}, \eqref{eqn:prime_power_contribution'}, and \eqref{eqn:prime_power_contribution_3} with the identity $\Lambda_{\pi}(\kp)=\lambda_{\pi}(\kp)\log\N\kp$.
\end{proof}

\subsection{Proof of \cref{thm:LFZDE}}

Our work in the \cref{subsec:detection} produces an upper bound for the count of zeros of $L(z,\pi)$ close to the line $\re(z)=1$.  See also the proof of \cite[Theorem 1.2]{ST}.

\begin{lemma}
\label{lem:zde_upper}
Under the notation and hypotheses of \cref{lem:prime_power_contribution}, if $w>2000$ is an absolute constant and $K \geq w\eta n\log(QT^{[F:\Q]})+O_{n,[F:\Q]}(1)$ with a sufficiently large implied constant, then
	\begin{equation}
	\label{eqn:5.6}
N_{\pi}(1-\eta/2,T)\ll (101)^{4K}K\eta^2 \int_{N_0}^{N_1}\int_{-T}^{T}\Big|\sum_{N_0\leq \N\kp\leq u}\frac{\lambda_{\pi}(\kp)\log\N\kp}{\N\kp^{1+i\tau}}\Big|^2 d\tau\frac{du}{u}.
\end{equation}
\end{lemma}
\begin{proof}
Recall our choice of $\eta$ in \eqref{eqn:eta_range}, and let $|\tau|\leq T$.  If $K$ is sufficiently large and $k\in[K,2K]$, then the $O$-term in \cref{lem:prime_power_contribution} is $\ll_{n,[F:\Q]} k(110)^{-k}\leq (100)^{-k-1}/4$.  As in \cref{lem:prime_power_contribution}, let $N_0=\exp(K/(300\eta))$ and $N_1=\exp(40K/\eta)$.  If $L(z,\rho)$ has a zero $\rho_0$ satisfying $|1+i\tau-\rho_0|\leq\eta$, then we combine \eqref{high_deriv_lower} and \cref{lem:zde_upper} to conclude our zero detection criterion: if $L(z,\rho)$ has a zero $\rho_0$ satisfying $|1+i\tau-\rho_0|\leq\eta$, then with our choice of $K$,
\begin{equation}
\label{eqn:zerodetect}
1\leq 4(100)^{2K+1}\eta^2\int_{N_0}^{N_1}\Big|\sum_{N_0\leq \N\kp\leq u}\frac{\lambda_{\pi}(\kp)\log\N\kp}{\N\kp^{1+i\tau}}\Big|\frac{du}{u}.
\end{equation}
We square both sides, apply Cauchy--Schwarz, and use \cref{lem:3.1} to deduce that
\[
\frac{\#\{\rho=\beta+i\gamma\colon \beta\geq 1-\eta/2,~|\gamma-\tau|\leq \eta/2\}}{\eta\log(C(\pi)T^{[F:\Q]})}\ll (101)^{4K}\eta^3\int_{N_0}^{N_1}\Big|\sum_{N_0\leq \N\kp\leq u}\frac{\lambda_{\pi}(\kp)\log\N\kp}{\N\kp^{1+i\tau}}\Big|^2\frac{du}{u}.
\]
To finish, we integrate over $|\tau|\leq T$ and observe that $\eta\log(C(\pi)T^{[F:\Q]})\ll K$.
\end{proof}

We use \cref{cor:MVT_primes} and \cref{lem:zde_upper} to prove \cref{thm:LFZDE}.

\begin{proof}[Proof of \cref{thm:LFZDE}]
First, let $n\geq 2$ so that if $\pi\in\mathfrak{F}_n(Q)$, then $L(s,\pi)$ is entire.  Choose $K = 10^5 n^4 \eta \log(QT^{[F:\Q]})+O_{n,[F:\Q]}(1)$ with a sufficiently large implied constant.  Recall $\eta$ from \eqref{eqn:eta_range} and the choices of $N_0$ and $N_1$.  We sum \eqref{eqn:5.6} over $\pi\in\mathfrak{F}_n(Q)$ to find that
	\begin{align}
		\label{eqn:estimate_0}
	\sum_{\pi\in\mathfrak{F}_{n}(Q)}N_{\pi}(1-\tfrac{\eta}{2},T)\ll (101)^{4K}K\eta^2 \int_{N_0}^{N_1}\sum_{\pi\in\mathfrak{F}_{n}(Q)}\int_{-T}^{T}\Big|\sum_{N_0\leq \N\kp\leq u}\frac{\lambda_{\pi}(\kp)\log\N\kp}{\N\kp^{1+i\tau}}\Big|^2 \frac{d\tau du}{u}.
	\end{align}
We apply \cref{cor:MVT_primes} with $y=N_0$ to deduce that \eqref{eqn:estimate_0} is $\ll (101)^{4K}K^3\ll (102)^{4K}$.  Using our choices of $K$ and $\eta$ and writing $\sigma=1-\eta/2$, we conclude that
\[
\sum_{\pi\in\mathfrak{F}_{n}(Q)}N_{\pi}(\sigma,T)\ll_{n,[F:\Q]} (QT^{[F:\Q]})^{10^7 n^4(1-\sigma)},\quad 1-\frac{1}{400n^3}\leq\sigma<1-\frac{1}{2\log(QT^{[F:\Q]})}.
\]
If $\sigma\geq 1-(2\log(QT^{[F:\Q]}))^{-1}$, then
\begin{align*}
\sum_{\pi\in\mathfrak{F}_{n}(Q)}N_{\pi}(\sigma,T)\leq \sum_{\pi\in\mathfrak{F}_{n}(Q)}N_{\pi}\Big(1-\frac{1}{2\log(QT^{[F:\Q]})},T\Big)\ll_{n,[F:\Q]} 1\ll_{n,[F:\Q]} (QT^{[F:\Q]})^{10^7 n^4(1-\sigma)}.
\end{align*}
On the other hand, if $\sigma<1-(400n^3)^{-1}$, then our estimate is trivial since $N_{\pi}(\frac{1}{2},T)\ll_{n,[F:\Q]} T\log(QT)$ for each $\pi\in\mathfrak{F}_{n}(Q)$ by \cite[Theorem 5.8]{IK}.

When $n=1$, our arguments do not change except that we omit the trivial representation, whose $L$-function is $\zeta_F(s)$ (which is not entire).  This $L$-function contributes negligibly since $N_{\pi}(\sigma,T)\ll_{[F:\Q]}(D_F T^{[F:\Q]})^{200(1-\sigma)}$ by \cite[Theorem 4.5]{TZ2}.  This completes the proof.
\end{proof}

\begin{proof}[Proof of \cref{thm:subconvexity}]
Let $\epsilon\geq 0$.  If $Q\ggg_{n,F} 1$, then $Q^{n}\ll |\mathfrak{F}_n(Q)|$ by \eqref{eqn:poly_upper}.  \cref{thm:LFZDE} implies that $N_{\pi}(1-\epsilon/(10^7n^3),6)=0$ for all except $\ll_{n,[F:\Q]} Q^{n\epsilon}\ll_{n,[F:\Q]} |\mathfrak{F}_n(Q)|^{\epsilon}$ of the $\pi\in\mathfrak{F}_n(Q)$. Since $\log|L(\frac{3}{2},\pi)|\ll_{n,[F:\Q]}1$ by \eqref{eqn:LRS_finite}, the theorem follows from \eqref{eqn:STl1/2}.
\end{proof}

\bibliographystyle{abbrv}
\bibliography{GeneralizedLinnik}

\def\cprime{$'$}
\begin{thebibliography}{10}

\bibitem{an2020logfree}
C.~{An}.
\newblock Log-free zero density estimates for automorphic $l$-functions, Apr.
  2020.
\newblock arXiv:2004.14410.

\bibitem{Blomer_twist}
V.~Blomer.
\newblock Subconvexity for twisted {$L$}-functions on {${\rm GL}(3)$}.
\newblock {\em Amer. J. Math.}, 134(5):1385--1421, 2012.

\bibitem{Blomer}
V.~{Blomer}.
\newblock {Density theorems for GL(n)}.
\newblock {\em arXiv e-prints}, page arXiv:1906.07459, June 2019.

\bibitem{Brumley}
F.~Brumley.
\newblock Effective multiplicity one on {${\rm GL}_N$} and narrow zero-free
  regions for {R}ankin-{S}elberg {$L$}-functions.
\newblock {\em Amer. J. Math.}, 128(6):1455--1474, 2006.

\bibitem{Brumley_2}
F.~Brumley.
\newblock Second order average estimates on local data of cusp forms.
\newblock {\em Arch. Math. (Basel)}, 87(1):19--32, 2006.

\bibitem{BM}
F.~{Brumley} and D.~{Mili{\'c}evi{\'c}}.
\newblock {Counting cusp forms by analytic conductor}.
\newblock {\em arXiv e-prints}, page arXiv:1805.00633, May 2018.

\bibitem{BTZ}
F.~{Brumley}, J.~{Thorner}, and A.~{Zaman}.
\newblock Zeros of {R}ankin--{S}elberg {$L$}-functions at the edge of the
  critical strip.
\newblock {\em arXiv e-prints}, page arXiv:1804.06402, Apr 2018.

\bibitem{Bump_lie}
D.~Bump.
\newblock {\em Lie groups}, volume 225 of {\em Graduate Texts in Mathematics}.
\newblock Springer, New York, second edition, 2013.

\bibitem{BH}
C.~J. Bushnell and G.~Henniart.
\newblock An upper bound on conductors for pairs.
\newblock {\em J. Number Theory}, 65(2):183--196, 1997.

\bibitem{DFI}
W.~Duke, J.~B. Friedlander, and H.~Iwaniec.
\newblock The subconvexity problem for {A}rtin {$L$}-functions.
\newblock {\em Invent. Math.}, 149(3):489--577, 2002.

\bibitem{DK}
W.~Duke and E.~Kowalski.
\newblock A problem of {L}innik for elliptic curves and mean-value estimates
  for automorphic representations.
\newblock {\em Invent. Math.}, 139(1):1--39, 2000.
\newblock With an appendix by Dinakar Ramakrishnan.

\bibitem{FI}
J.~Friedlander and H.~Iwaniec.
\newblock {\em Opera de cribro}, volume~57 of {\em American Mathematical
  Society Colloquium Publications}.
\newblock American Mathematical Society, Providence, RI, 2010.

\bibitem{Gallagher}
P.~X. Gallagher.
\newblock A large sieve density estimate near {$\sigma =1$}.
\newblock {\em Invent. Math.}, 11:329--339, 1970.

\bibitem{HarcosMIchel}
G.~Harcos and P.~Michel.
\newblock The subconvexity problem for {R}ankin-{S}elberg {$L$}-functions and
  equidistribution of {H}eegner points. {II}.
\newblock {\em Invent. Math.}, 163(3):581--655, 2006.

\bibitem{Humphries}
P.~Humphries and F.~Brumley.
\newblock Standard zero-free regions for {R}ankin--{S}elberg {L}-functions via
  sieve theory.
\newblock {\em Math. Z.}, 292(3-4):1105--1122, 2019.

\bibitem{IK}
H.~Iwaniec and E.~Kowalski.
\newblock {\em Analytic number theory}, volume~53 of {\em American Mathematical
  Society Colloquium Publications}.
\newblock American Mathematical Society, Providence, RI, 2004.

\bibitem{IS}
H.~Iwaniec and P.~Sarnak.
\newblock Perspectives on the analytic theory of {$L$}-functions.
\newblock {\em Geom. Funct. Anal.}, (Special Volume, Part II):705--741, 2000.
\newblock GAFA 2000 (Tel Aviv, 1999).

\bibitem{2020arXiv200109640J}
S.~{Jana}.
\newblock {Applications of analytic newvectors for $\mathrm{GL}(r)$}.
\newblock {\em arXiv e-prints}, page arXiv:2001.09640, Jan. 2020.

\bibitem{Jutila}
M.~Jutila.
\newblock On {L}innik's constant.
\newblock {\em Math. Scand.}, 41(1):45--62, 1977.

\bibitem{Kim}
H.~H. Kim.
\newblock Functoriality for the exterior square of {${\rm GL}_4$} and the
  symmetric fourth of {${\rm GL}_2$}.
\newblock {\em J. Amer. Math. Soc.}, 16(1):139--183, 2003.
\newblock With appendix 1 by Dinakar Ramakrishnan and appendix 2 by Kim and
  Peter Sarnak.

\bibitem{KM}
E.~Kowalski and P.~Michel.
\newblock Zeros of families of automorphic {$L$}-functions close to 1.
\newblock {\em Pacific J. Math.}, 207(2):411--431, 2002.

\bibitem{Li}
X.~Li.
\newblock Upper bounds on {$L$}-functions at the edge of the critical strip.
\newblock {\em Int. Math. Res. Not. IMRN}, (4):727--755, 2010.

\bibitem{Li_Xiaoqing}
X.~Li.
\newblock Bounds for {${\rm GL}(3)\times {\rm GL}(2)$} {$L$}-functions and
  {${\rm GL}(3)$} {$L$}-functions.
\newblock {\em Ann. of Math. (2)}, 173(1):301--336, 2011.

\bibitem{Linnik}
U.~V. Linnik.
\newblock On the least prime in an arithmetic progression.
\newblock {\em Rec. Math. [Mat. Sbornik] N.S.}, 15(57):139--178,347--368, 1944.

\bibitem{LRS}
W.~Luo, Z.~Rudnick, and P.~Sarnak.
\newblock On the generalized {R}amanujan conjecture for {${\rm GL}(n)$}.
\newblock In {\em Automorphic forms, automorphic representations, and
  arithmetic ({F}ort {W}orth, {TX}, 1996)}, volume~66 of {\em Proc. Sympos.
  Pure Math.}, pages 301--310. Amer. Math. Soc., Providence, RI, 1999.

\bibitem{Michel}
P.~Michel.
\newblock Analytic number theory and families of automorphic {$L$}-functions.
\newblock In {\em Automorphic forms and applications}, volume~12 of {\em
  IAS/Park City Math. Ser.}, pages 181--295. AMS, Providence, RI, 2007.

\bibitem{MV}
P.~Michel and A.~Venkatesh.
\newblock The subconvexity problem for {${\rm GL}_2$}.
\newblock {\em Publ. Math. Inst. Hautes \'{E}tudes Sci.}, (111):171--271, 2010.

\bibitem{Montgomery}
H.~L. Montgomery.
\newblock Zeros of {$L$}-functions.
\newblock {\em Invent. Math.}, 8:346--354, 1969.

\bibitem{MS}
W.~M{\"u}ller and B.~Speh.
\newblock Absolute convergence of the spectral side of the {A}rthur trace
  formula for {${\rm GL}_n$}.
\newblock {\em Geom. Funct. Anal.}, 14(1):58--93, 2004.
\newblock With an appendix by E. M. Lapid.

\bibitem{Munshi}
R.~Munshi.
\newblock The circle method and bounds for {$L$}-functions---{III}:
  {$t$}-aspect subconvexity for {$GL(3)$} {$L$}-functions.
\newblock {\em J. Amer. Math. Soc.}, 28(4):913--938, 2015.

\bibitem{Turan}
V.~T. S{\'o}s and P.~Tur{\'a}n.
\newblock On some new theorems in the theory of {D}iophantine approximations.
\newblock {\em Acta Math. Acad. Sci. Hungar.}, 6:241--255, 1955.

\bibitem{ST}
K.~Soundararajan and J.~Thorner.
\newblock Weak subconvexity without a {R}amanujan hypothesis.
\newblock {\em Duke Math. J.}, 168:1231--1268, 2019.
\newblock With an appendix by Farrell Brumley.

\bibitem{TZ1}
J.~Thorner and A.~Zaman.
\newblock An explicit bound for the least prime ideal in the {C}hebotarev
  density theorem.
\newblock {\em Algebra Number Theory}, 11(5):1135--1197, 2017.

\bibitem{TZ2}
J.~Thorner and A.~Zaman.
\newblock A {C}hebotarev variant of the {B}run-{T}itchmarsh theorem and bounds
  for the {L}ang-{T}rotter conjectures.
\newblock {\em Int. Math. Res. Not. IMRN}, (16):4991--5027, 2018.

\bibitem{Venkatesh}
A.~Venkatesh.
\newblock Large sieve inequalities for {${\rm GL}(n)$}-forms in the conductor
  aspect.
\newblock {\em Adv. Math.}, 200(2):336--356, 2006.

\bibitem{Venk1}
A.~Venkatesh.
\newblock Sparse equidistribution problems, period bounds and subconvexity.
\newblock {\em Ann. of Math. (2)}, 172(2):989--1094, 2010.

\bibitem{Weiss}
A.~Weiss.
\newblock The least prime ideal.
\newblock {\em J. Reine Angew. Math.}, 338:56--94, 1983.

\end{thebibliography}

\end{document}